\documentclass[11pt]{amsart}

\usepackage[cp1250]{inputenc}
\usepackage{graphicx}
\usepackage{url}
\usepackage[all]{xy}
\usepackage{multicol}
\usepackage{amsthm}
\usepackage{amsmath}
\usepackage{amssymb}
\usepackage{geometry}
\usepackage{pinlabel}
\usepackage{hyperref}
\usepackage{array}
\usepackage{xcolor}
\usepackage{mathtools}
\usepackage{tikz}
\usetikzlibrary{positioning}
\usepackage{pdflscape}
\usepackage{listofitems}
\usepackage{float}

\newtheorem{Theorem}{Theorem}[section]
\newtheorem*{theorem}{Theorem}
\newtheorem{lemma}[Theorem]{Lemma}
\newtheorem{proposition}[Theorem]{Proposition}
\newtheorem{corollary}[Theorem]{Corollary}
\newtheorem*{definition}{Definition}
\newtheorem{remark}[Theorem]{Remark}
\newtheorem{example}[Theorem]{Example}

\newcommand\cycle[2][\,]{%
  \readlist\thecycle{#2}%
  (\foreachitem\i\in\thecycle{\ifnum\icnt=1\else#1\fi\i})%
}


\begin{document}

\title{On discrete symmetries of the cube of smoothings}

\author{Eva Horvat}
\address{University of Ljubljana, Faculty of Education, Kardeljeva plo\v s\v cad 16, 1000 Ljubljana, Slovenia, eva.horvat@pef.uni-lj.si}
\email{\href{mailto:eva.horvat@pef.uni-lj.si}{eva.horvat@pef.uni-lj.si}}

\begin{abstract} A link diagram with a labeled barycentric subdivision endows the cube of resolutions with an additional combinatorial structure. We study the set of symmetries, preserving this structure. We construct a combinatorial link invariant, based on the symmetries of triangulated smoothings.  
\end{abstract}

\maketitle

\section{Introduction}

The cube of resolutions of a link diagram represents an important combinatorial tool that has been used to construct powerful link invariants. The most famous of these are Khovanov homology, that categorifies the Jones polynomial of links \cite{KH}, Khovanov-Rozansky homology, that categorifies the HOMFLY-PT polynomial \cite{KR1,KR2}, knot Floer homology, developed by Ozsv\'{a}th and Szab\'{o} \cite{OS}, and the $s$-invariant of Rasmussen \cite{RA}. Dror Bar-Natan extended Khovanov homology to tangles by constructing a geometric complex, based on the cube of resolutions, and establishing its invariance \cite{BN}. 

Recent years have seen a resurge of activity in generalizations and extensions of theories mentioned above. The local nature of the geometric complex allowed the development of high-performance algorithms for homology computations \cite{BN1, KE}. Bar-Natan's cobordism-based category provided a template for categorifying other knot invariants, e. g. the odd annular Bar-Natan category \cite{NW} and Lie superalgebra categorifications. In \cite{SG}, a version of Khovanov homology for alternating links with marking data, inspired by instanton theory, is introduced. 

In this paper, we endow a link diagram with a labeled barycentric subdivision and study the combinatorial structure, induced on its cube of resolutions. Such cube admits a well-defined orientation of all its triangulated smoothings. We define a colored 1-dimensional complex, called the web of a triangulated diagram, that contains all the resolutions, and present the set of web moves. These diagrammatic tools allow us to study symmetries of triangulated smoothings. We define a cube of permutation data, based on a triangulated link diagram with a chosen web orientation. Constructing an equivalence relation on the set of labeled hypercubes, we prove our main result: 
\begin{theorem} The equivalence class of a cube of permutation data is an invariant of the underlying link. 
\end{theorem}

The paper is organized as follows. In Section \ref{sec1}, we review some basics on the construction of the Khovanov complex. Section \ref{sec2} defines diagram triangulations and the induced combinatorial data in the cube of smoothings. In Subsection \ref{subs21}, we present the construction of triangulated link diagrams, triangulated cube of smoothings and webs, and describe their basic properties. In Subsection \ref{subs22}, we study symmetries of triangulated smoothings. To each vertex $\mathbf{v}$ in the cube of smoothings, we associate a group $G_{\mathbf{v}}$, and a coordinated choice of orientations of all smoothings allows a well-defined vertex labeling by elements $\sigma _{\mathbf{v}}\in G_{\mathbf{v}}$. The resulting labeled hypercube is called a cube of permutation data, associated with a triangulated link diagram. We define an equivalence relation on the set of hypercubes with a vertex labeling, and show that the equivalence class of a cube of permutation data gives a link invariant.  

\section{Preliminaries} \label{sec1}

We review some basics on the construction of the Khovanov complex. A good introduction to Khovanov homology can be found in \cite{BN} or \cite{TUR}.

Suppose we have a diagram $\mathcal{D}$ of an oriented link $L$. Denote by $k_+$ (resp. $k_{-}$) the number of crossings with a positive (resp. negative) sign and let $k=k_{+}+k_{-}$.  At every crossing, the diagram may be locally transformed (resolved) by a 0-\emph{smoothing} or a 1-\emph{smoothing}, see Figure \ref{fig1}.  
\begin{figure}[H]
\begin{tikzpicture}
\matrix
{
\node {\begin{tikzpicture}[scale=0.50]
\draw[black, dashed] (0,0) circle (40pt); 
\draw[black, thick][-] (-1,-1) -- (1,1);
\draw[black, thick] (-1,1) -- (-0.1,0.1);
\draw[black, thick][-] (0.1,-0.1) -- (1,-1);
\end{tikzpicture}}; & \node {\begin{tikzpicture}[scale=0.50]
\draw[black, dashed] (0,0) circle (40pt); 
\draw[black, thick][-] (-1,-1) .. controls (-0.3,0) .. (-1,1) ;
\draw[black, thick][-] (1,-1) .. controls (0.3,0) .. (1,1);
\end{tikzpicture}};
& \node {\begin{tikzpicture}[scale=0.50]
\draw[black, dashed] (0,0) circle (40pt); 
\draw[black, thick][-] (-1,1) .. controls (0,0.3) .. (1,1) ;
\draw[black, thick] (-1,-1) .. controls (0,-0.3) .. (1,-1);
\end{tikzpicture}};  \\
};
\end{tikzpicture}
\caption{A crossing (left), its 0-smoothing (center) and its 1-smoothing (right)}
\label{fig1}
\end{figure}

Applying one of the two resolutions at every crossing of the diagram $\mathcal{D}$, we are left with a disjoint union of circles (an unlink diagram) that we call a \emph{resolution} or a \emph{smoothing} of $\mathcal{D}$. A $k$-crossing diagram has $2^{k}$ smoothings, labeled by words $\mathbf{v}\in \{0,1\}^{k}$. Here $\mathbf{v}_{l}=1$ (resp. $\mathbf{v}_{l}=0$) if at the $l$-th crossing, the 1-smoothing (resp. the 0-smoothing) has been applied.

\begin{figure}[H] 
\begin{tikzpicture}
\matrix [draw,column sep={1cm}]
{
\node {}; & \node (a) {$001$}; & \node (b) {$011$}; & \node{};\\
\node (c) {$000$}; & \node (d) {$010$}; & \node (e) {$101$}; & \node (f) {$111$};\\
\node {}; & \node (g) {$100$}; & \node (h) {$110$}; & \node{};\\
};
\draw [blue] (a.east) -- (b.west) node [above,midway] {};
\draw [blue] (a.east) -- (e.west) node [above,midway] {};
\draw [blue] (c.east) -- (a.west) node [above,midway] {};
\draw [blue] (c.east) -- (d.west) node [above,midway] {};
\draw [blue] (c.east) -- (g.west) node [above,midway] {};
\draw [blue] (d.east) -- (b.west) node [above,midway] {};
\draw [blue] (d.east) -- (h.west) node [above,midway] {};
\draw [blue] (g.east) -- (e.west) node [above,midway] {};
\draw [blue] (e.east) -- (f.west) node [above,midway] {};
\draw [blue] (b.east) -- (f.west) node [above,midway] {};
\draw [blue] (h.east) -- (f.west) node [above,midway] {};
\draw [blue] (g.east) -- (h.west) node [above,midway] {};
\end{tikzpicture}
\caption{The cube of smoothings for a diagram with 3 crossings}
\label{fig2}
\end{figure}

The set $\{0,1\}^{k}$  is the vertex set of a $k$-dimensional hypercube $Q_k$, with edges between every two words that differ exactly at one place. We draw the cube in a skewered manner, so that all vertices whose sum of coordinates equals $i$ have the same $x$-coordinate $i-k_{-}$ (see Figure \ref{fig2}). Let $\Gamma _{\mathbf{v}}$ denote the smoothing labeled by a word $\mathbf{v}$, and let $c_{\mathbf{v}}$ denote the number of components of $\Gamma _{\mathbf{v}}$. Denote by $r_{\mathbf{v}}$ the number of $1$'s in the word $\mathbf{v}$. The unnormalised Jones polynomial of the link with a diagram $\mathcal{D}$ is given by $$\widehat{J}(\mathcal{D})(q)=\sum _{\mathbf{v}\in \{0,1\}^{k}}(-1)^{r_{\mathbf{v}}+k_{-}}\,q^{r_{\mathbf{v}}+k_{+}-2k_{-}}\left (q+q^{-1}\right )^{c_{\mathbf{v}}}\;.$$ The ``usual'' Jones polynomial may be obtained from the Jones polynomial $J(\mathcal{D})(q)=\frac{\widehat{J}(\mathcal{D})(q)}{q+q^{-1}}$ by the substitution $q=-t^{\frac{1}{2}}$.

An edge between two vertices in the cube of smoothings is usually labelled by a word of ones and zeroes with a star at the position where the two vertices differ. Thus for example the edge joining the vertex $010$ and $011$ is labeled by $01\star$. Every edge is oriented as an arrow from the vertex with $\star =0$ to the vertex with $\star =1$. The smoothings labelled by two adjacent vertices $\mathbf{v}\stackrel{\zeta }{\to}\mathbf{v}'$ only differ inside the region of a small disc (the changing disc) around the crossing, corresponding to the $\star $ in $\zeta $, where the $0$-smoothing switches to the $1$-smoothing. The edge $\zeta $ corresponds to a cobordism $W_{\zeta }$ between the smoothings. Outside the changing disc, this cobordism is just a product of $\Gamma _{\mathbf{v}}$ with an interval, while inside the tube above the changing disc, the surface looks like a saddle between the two resolutions.

\section{A discrete structure on the cube of smoothings} \label{sec2}

\subsection{Triangulated diagrams and webs} \label{subs21}

Let $\mathcal{D}$ be a diagram of a link $L$ with $k$ crossings. $\mathcal{D}$ may contain crossingless unlinked components, which we will call \textit{empty} components; denote by $\mu $ the number of empty components. By choosing a point (vertex) on each empty component of $\mathcal{D}$, $\mathcal{D}$ becomes a 1-complex, i.e. a graph. 

Denote by $E(\mathcal{D})$ the set of edges of this graph, let $V=\{1,2,\ldots ,2k+\mu \}$ and choose a bijection $\tau\colon E(\mathcal{D})\to V$. For any $i\in \{1,2,\ldots ,2k+\mu \}$, we define $$N(i)=\{\tau (e')\,|\,\textrm{$\tau ^{-1}(i)$ and $e'$ are distinct adjacent edges of $\mathcal{D}$ that belong to the same component of $L$}\}\;.$$ 
Observe that $|N(i)|\leq 2$ for all $i$, and that $|N(i)|=0$ iff $i$ is the label of an empty component of $\mathcal{D}$. 
A barycentric subdivision of the diagram $\mathcal{D}$ adds a new vertex in the center of every edge $e\in E(\mathcal{D})$; we label this subdivision vertex by $\tau (e)\in V$. The collection of labeled subdivision vertices on every edge $e\in E(\mathcal{D})$ will be called a \emph{triangulation} of the diagram $\mathcal{D}$. A link diagram $\mathcal{D}$ with a chosen triangulation $\tau $ will be called a \emph{triangulated link diagram} $(\mathcal{D},\tau )$, see Figure \ref{fig3}.   

\begin{figure}[H]
\begin{tikzpicture}
\matrix [column sep=1.2cm]
{
\node (a) {\begin{tikzpicture}
\draw[black, thick] (-0.07738,0.9064) arc (65:413:1); 
\draw[black, thick] (0.12539,-0.9271839) arc (-112:233:1); 
\filldraw[black] (-1.5,0) circle (2pt) node[anchor=east]{$1$};
\filldraw[black] (-0.5,0) circle (2pt) node[anchor=east]{$2$};
\filldraw[black] (1.5,0) circle (2pt) node[anchor=west]{$4$};
\filldraw[black] (0.5,0) circle (2pt) node[anchor=west]{$3$};
\end{tikzpicture}}; & \node (b) {\begin{tikzpicture}
\filldraw[black] (-5,0) circle (2pt) node[anchor=east]{$1$};
\filldraw[black] (-4,0) circle (2pt) node[anchor=east]{$2$};
\filldraw[black] (-2.5,0) circle (2pt) node[anchor=west]{$4$};
\filldraw[black] (-3.5,0) circle (2pt) node[anchor=west]{$3$};
\draw[black, thick] (-4.5,0) circle (0.5cm);
\draw[black, thick][-] (-3,0) circle (0.5cm);
\filldraw[blue] (-3.75,0.6) node[anchor=south]{$00$};
\filldraw[black] (-1.25,1.5) circle (2pt) node[anchor=east]{$1$};
\filldraw[black] (-0.25,1.5) circle (2pt) node[anchor=east]{$2$};
\filldraw[black] (1.25,1.5) circle (2pt) node[anchor=west]{$4$};
\filldraw[black] (0.25,1.5) circle (2pt) node[anchor=west]{$3$};
\filldraw[blue] (0,2) node[anchor=south]{$01$};
\filldraw[blue] (0,-2.4) node[anchor=south]{$10$};
\draw[black, thick] (-1.25,1.5) .. controls (-1.1,0.3) and (1.1,0.3) .. (1.25,1.5);
\draw[black, thick] (-1.25,1.5) .. controls (-1,2) and (-0.5,2) .. (-0.25,1.5);
\draw[black, thick] (1.25,1.5) .. controls (1,2) and (0.5,2) .. (0.25,1.5);
\draw[black, thick] (-0.25,1.5) .. controls (-0.15,1.2) and (0.15,1.2) .. (0.25,1.5);
\filldraw[black] (-1.25,-1.5) circle (2pt) node[anchor=east]{$1$};
\filldraw[black] (-0.25,-1.5) circle (2pt) node[anchor=east]{$2$};
\filldraw[black] (1.25,-1.5) circle (2pt) node[anchor=west]{$4$};
\filldraw[black] (0.25,-1.5) circle (2pt) node[anchor=west]{$3$};
\draw[black, thick] (-1.25,-1.5) .. controls (-1.1,-0.3) and (1.1,-0.3) .. (1.25,-1.5);
\draw[black, thick] (-1.25,-1.5) .. controls (-1,-2) and (-0.5,-2) .. (-0.25,-1.5);
\draw[black, thick] (1.25,-1.5) .. controls (1,-2) and (0.5,-2) .. (0.25,-1.5);
\draw[black, thick] (-0.25,-1.5) .. controls (-0.15,-1.2) and (0.15,-1.2) .. (0.25,-1.5);
\filldraw[black] (2.5,0) circle (2pt) node[anchor=east]{$1$};
\filldraw[black] (3.15,0) circle (2pt) node[anchor=east]{$2$};
\filldraw[black] (4.5,0) circle (2pt) node[anchor=west]{$4$};
\filldraw[black] (3.85,0) circle (2pt) node[anchor=west]{$3$};
\filldraw[blue] (3.5,0.8) node[anchor=south]{$11$};
\draw[black, thick][-] (2.5,0).. controls (2.7,1) and (4.3,1) .. (4.5,0);         
\draw[black, thick][-] (2.5,0).. controls (2.7,-1) and (4.3,-1) .. (4.5,0);     
\draw[black, thick][-] (3.5,0) circle (10pt);   
\draw[blue, thick][->] (-2,0.5)--(-1.5,1);         
\draw[blue, thick][->] (-2,-0.5)--(-1.5,-1);       
\draw[blue, thick][->] (1.5,1)--(2,0.5);       
\draw[blue, thick][->] (1.5,-1)--(2,-0.5);                 
\end{tikzpicture}};\\
};
\end{tikzpicture}
\caption{A triangulated link diagram (left) and its cube of smoothings (right)}
\label{fig3}
\end{figure}

Every smoothing $\Gamma _{\mathbf{v}}$ of a triangulated link diagram has an induced triangulation by labeled vertices, so that $\Gamma _{\mathbf{v}}$ is a graph with a vertex set $V$ and an edge set $E_{\mathbf{v}}$. For any $\mathbf{v}\in \{0,1\}^{k}$, connected components of the graph $\Gamma _{\mathbf{v}}$ induce a partition $\mathcal{P}_{\mathbf{v}}$ of the vertex set $V$. An edge of the cube of smoothings corresponds to a cobordism, given by band addition. All edges, whose labels carry a star at the $m$-th position, correspond to addition of the same rectangular band that resolves the $m$-th crossing of $\mathcal{D}$, see Figure \ref{fig4}. An edge $\mathbf{v}\stackrel{\zeta}{\rightarrow}\mathbf{v}'$ transforms the graph $(V,E_{\mathbf{v}})$ into $(V,E_{\mathbf{v}'})$ by surgery along a band which either merges two sets of the partition $\mathcal{P}_{\mathbf{v}}$ or splits a set of $\mathcal{P}_{\mathbf{v}}$ into a disjoint union of two subsets. 

\begin{figure}[H]
\begin{tikzpicture}
\matrix
{
\node {\begin{tikzpicture}[scale=0.50]
\draw[black, dashed] (0,0) circle (40pt); 
\draw[black, thick][-] (-1,-1) -- (1,1);
\draw[black, thick] (-1,1) -- (-0.1,0.1);
\draw[black, thick][-] (0.1,-0.1) -- (1,-1);
\filldraw[black] (-1,1) circle (2pt) node[anchor=east]{$v$};
\filldraw[black] (1,1) circle (2pt) node[anchor=west]{$j$};
\filldraw[black] (-1,-1) circle (2pt) node[anchor=east]{$i$};
\filldraw[black] (1,-1) circle (2pt) node[anchor=west]{$w$};
\end{tikzpicture}}; & \node {\begin{tikzpicture}[scale=0.50]
\draw[black, dashed] (0,0) circle (40pt); 
\draw[black, thick][-] (-1,-1) .. controls (-0.3,0) .. (-1,1) ;
\draw[black, thick][-] (1,-1) .. controls (0.3,0) .. (1,1);
\filldraw[black] (-1,1) circle (2pt) node[anchor=east]{$v$};
\filldraw[black] (1,1) circle (2pt) node[anchor=west]{$j$};
\filldraw[black] (-1,-1) circle (2pt) node[anchor=east]{$i$};
\filldraw[black] (1,-1) circle (2pt) node[anchor=west]{$w$};
\end{tikzpicture}}; 
& \node {\begin{tikzpicture}[scale=0.50]
\fill[blue!40!white,draw=black] (-1,-1.5) rectangle (1,-0.5); 
\draw[blue, thick][->] (-1,0) -- (1,0);
\filldraw[black] (-1,-1.5) circle (2pt) node[anchor=east]{$i$};
\filldraw[black] (-1,-0.5) circle (2pt) node[anchor=east]{$v$};
\filldraw[black] (1,-0.5) circle (2pt) node[anchor=west]{$j$};
\filldraw[black] (1,-1.5) circle (2pt) node[anchor=west]{$w$};
\end{tikzpicture}};
& \node {\begin{tikzpicture}[scale=0.50]
\draw[black, dashed] (0,0) circle (40pt); 
\draw[black, thick][-] (-1,1) .. controls (0,0.3) .. (1,1) ;
\draw[black, thick][-] (-1,-1) .. controls (0,-0.3) .. (1,-1);
\filldraw[black] (-1,1) circle (2pt) node[anchor=east]{$v$};
\filldraw[black] (1,1) circle (2pt) node[anchor=west]{$j$};
\filldraw[black] (-1,-1) circle (2pt) node[anchor=east]{$i$};
\filldraw[black] (1,-1) circle (2pt) node[anchor=west]{$w$};
\end{tikzpicture}}; \\
};
\end{tikzpicture}
\caption{A crossing of $\mathcal{D}$ (left) and the band, connecting its two resolutions (right)}
\label{fig4}
\end{figure}

Denote by $\mathbf{0}\in \{0,1\}^{k}$ (resp. $\mathbf{1}\in \{0,1\}^{k}$) a $k$-tuple of zeroes (resp. ones). Consider the graph $\Gamma =(V,E_{\mathbf{0}}\cup E_{\mathbf{1}})$ with a coloring function on the set of edges $\gamma \colon E_{\mathbf{0}}\cup E_{\mathbf{1}}\to \{0,1,2\}$, defined by $$\gamma (e)=\begin{cases} 0 & \textrm{ if $e\in E_{\mathbf{0}}\backslash E_{\mathbf{1}}\;,$}\\ 1 & \textrm{ if $e\in E_{\mathbf{1}}\backslash E_{\mathbf{0}}\;,$}\\ 2 & \textrm{ if $e\in E_{\mathbf{0}}\cap E_{\mathbf{1}}\;.$}\end{cases}$$ 
The colored 1-complex $(\Gamma ,\gamma )$ will be called the \emph{web} of the triangulated diagram $(\mathcal{D},\tau )$. Observe that $\Gamma $ is a planar graph; we denote by $\Sigma $ its ambient plane. The closure of a connected component of $\Sigma -\Gamma $ is called a \emph{region}. Note that $k$ of these regions are bands, whose boundary contains a pair of opposite edges of color $0$ and another pair of color $1$. 

\begin{example} Figure \ref{fig5} shows the web of a triangulated diagram of the Hopf link. 
\end{example}

\begin{figure}[H]
\begin{tikzpicture}
\matrix [column sep=1.5cm]
{
\node (a) {\begin{tikzpicture}
\draw[black, thick] (-0.07738,0.9064) arc (65:413:1); 
\draw[black, thick] (0.12539,-0.9271839) arc (-112:233:1); 
\filldraw[black] (-1.5,0) circle (2pt) node[anchor=east]{$1$};
\filldraw[black] (-0.5,0) circle (2pt) node[anchor=east]{$2$};
\filldraw[black] (1.5,0) circle (2pt) node[anchor=west]{$4$};
\filldraw[black] (0.5,0) circle (2pt) node[anchor=west]{$3$};
\end{tikzpicture}}; & \node (b) {\begin{tikzpicture}
\filldraw[black] (-1.5,0) circle (2pt) node[anchor=east]{$1$};
\filldraw[black] (-0.5,0) circle (2pt) node[anchor=east]{$2$};
\filldraw[black] (1.5,0) circle (2pt) node[anchor=west]{$4$};
\filldraw[black] (0.5,0) circle (2pt) node[anchor=west]{$3$};
\draw[blue, thick] (-1,0) circle (0.5cm);
\draw[blue, thick][-] (1,0) circle (0.5cm);
\draw[red, thick][-] (-1.5,0).. controls (-1.5,1.6) and (1.5,1.6) .. (1.5,0); 
\draw[red, thick][-] (-1.5,0).. controls (-1.5,-1.6) and (1.5,-1.6) .. (1.5,0); 
\draw[red, thick][-] (0,0) circle (0.5cm);                    
\end{tikzpicture}};\\
};
\end{tikzpicture}
\caption{A triangulated link diagram (left) and its corresponding web (right).}
\label{fig5}
\end{figure}

\begin{lemma} \label{lemma1} An orientation of the plane $\Sigma $ induces a uniquely defined orientation of every smoothing of $\mathcal{D}$.  
\end{lemma}
\begin{proof} Let $(\Gamma ,\gamma )$ be the web of a triangulated diagram $\mathcal{D}$, and fix an orientation of its ambient plane $\Sigma $. Let us denote by $b_{i}$ the band region that corresponds to the $i$-th crossing of the diagram $\mathcal{D}$ for $i=1,2,\ldots ,k$. Every smoothing $\Gamma _{\mathbf{v}}$ is the boundary of a compact subset $R_{\mathbf{v}}\subset \Sigma $ that is a union of regions. For any $\mathbf{v}\in \{0,1\}^{k}$, the set $R_{\mathbf{v}}$ is obtained from $R_{\mathbf{0}}$ by surgery along all bands $b_{i}$, for which $\mathbf{v}_{i}=1$. Each component of $R_{\mathbf{v}}$ inherits the orientation of the plane $\Sigma $, which uniquely defines an orientation of its boundary $\Gamma _{\mathbf{v}}$. 
\end{proof}

A collection of orientations of all smoothings $\Gamma _{\mathbf{v}}$ for $\mathbf{v}\in \{0,1\}^{k}$, induced by a fixed orientation of $\Sigma $ as in Lemma \ref{lemma1}, will be called a \emph{web orientation}. Observe that the cube of smoothings allows two web orientations.\\ 

In order to search for link invariants, arising from webs of triangulated diagrams, we need to study invariance of this construction. Our choice of the labeling function $\tau \colon E(\mathcal{D})\to V$ could have been replaced by any other bijection, thus a permutation of  labels should result in the same value of any possible invariant. 

Two diagrams of equivalent links are connected by a finite sequence of plane isotopies and Reidemeister moves, see Figure \ref{fig11}. During the first Reidemeister move of a triangulated diagram $(\mathcal{D},\tau )$, a new crossing appears, which gives rise to a new pair of labeled vertices in the barycentric subdivision. Figure \ref{fig6} shows the resulting pair of local smoothings (which of them is the $0$-smoothing, depends on the type of move - $R_1$ or $R_{1}'$). During the second Reidemeister move, a pair of new crossings between the edges labeled by $i$ and $j$ give rise to a quadruple of new vertices $a$, $b$, $c$, $d$; the resulting local smoothing is given in Figure \ref{fig8}. Two local smoothings of a triangulated diagram, connected by the third Reidemeister move, are shown in Figure \ref{fig9}. Figure \ref{fig7} shows the set of moves on the web of a triangulated diagram $(\mathcal{D},\tau)$, induced by the Reidemeister moves of $\mathcal{D}$.  

\begin{figure}[H]
\begin{tikzpicture}
\matrix [column sep=1cm]
{
\node (c) {\begin{tikzpicture}[scale=0.80]
\draw[black, dashed] (0,0) circle (40pt);
\draw[black, thick][-] (-1,1) .. controls (0.3,0.5) and (0.7,-0.2) .. (0,-0.4) 
                                           .. controls (-0.7,-0.2) and (-0.3,0.2) .. (0,0.4);
\draw[black,thick][-] (0.1,0.45) .. controls (0.5,0.7) .. (1,1);
\filldraw[black] (0,-0.4) circle (2pt) node[anchor=north]{$i$};
\filldraw[black] (-0.5,0.76) circle (2pt) node[anchor=south]{$a$};
\filldraw[black] (0.5,0.7) circle (2pt) node[anchor=south]{$b$};
\end{tikzpicture}}; & 
\node (d) {\begin{tikzpicture}[scale=0.80]
\draw[black, dashed] (0,0) circle (40pt);
\draw[black, thick][-] (-1,1) .. controls (-0.3,0) and (0.3,0) .. (1,1) ;
\filldraw[black] (0,0.27) circle (2pt) node[anchor=north]{$i$};
\end{tikzpicture}}; & 
\node (j) {\begin{tikzpicture}[scale=0.80]
\draw[black, dashed] (0,0) circle (40pt);
\draw[black, thick][-] (1,1) .. controls (-0.3,0.5) and (-0.7,-0.2) .. (0,-0.4) 
                                           .. controls (0.7,-0.2) and (0.3,0.2) .. (0,0.4);
\draw[black,thick][-] (-0.1,0.45) .. controls (-0.5,0.7) .. (-1,1);
\filldraw[black] (0,-0.4) circle (2pt) node[anchor=north]{$i$};
\filldraw[black] (-0.55,0.74) circle (2pt) node[anchor=south]{$a$};
\filldraw[black] (0.5,0.77) circle (2pt) node[anchor=south]{$b$};
\end{tikzpicture}}; \\
\node (e) {\begin{tikzpicture}[scale=0.90]
\draw[black, dashed] (0,0) circle (40pt);
\draw[black, thick][-] (-1,1) .. controls (0.6,0.4) and (0.6,-0.4) .. (-1,-1) ;
\draw[black, thick][-] (1,1) .. controls (0.5,0.9) and (0.3,0.6) .. (0,0.5); 
\draw[black,thick][-] (1,-1) .. controls (0.5,-0.9) and (0.3,-0.6) .. (0,-0.5); 
\draw[black,thick][-] (-0.1,-0.45) .. controls (-0.7,-0.35) and (-0.7,0.35) .. (-0.1,0.45);
\filldraw[black] (0.5,-0.79) circle (2pt) node[anchor=north]{$d$};
\filldraw[black] (0.5,0.79) circle (2pt) node[anchor=south]{$b$};
\filldraw[black] (-0.5,-0.77) circle (2pt) node[anchor=north]{$c$};
\filldraw[black] (-0.5,0.77) circle (2pt) node[anchor=south]{$a$};
\filldraw[black] (-0.53,0) circle (2pt) node[anchor=east]{$i$};
\filldraw[black] (0.2,0) circle (2pt) node[anchor=west]{$j$};
\end{tikzpicture}}; &
\node (f) {\begin{tikzpicture}[scale=0.90]
\draw[black, dashed] (0,0) circle (40pt);
\draw[black, thick][-] (-1,1) .. controls (-0.15,0.5) and (-0.15,-0.5) .. (-1,-1) ;
\draw[black, thick][-] (1,1) .. controls (0.15,0.5) and (0.15,-0.5) .. (1,-1) ;
\filldraw[black] (-0.35,0) circle (2pt) node[anchor=east]{$j$};
\filldraw[black] (0.35,0) circle (2pt) node[anchor=west]{$i$};
\end{tikzpicture}}; &
\node (g) {\begin{tikzpicture}[scale=0.90]
\draw[black, dashed] (0,0) circle (40pt);
\draw[black, thick][-] (1,1) .. controls (-0.6,0.4) and (-0.6,-0.4) .. (1,-1) ;
\draw[black, thick][-] (-1,1) .. controls (-0.5,0.9) and (-0.3,0.6) .. (0,0.5); 
\draw[black,thick][-] (-1,-1) .. controls (-0.5,-0.9) and (-0.3,-0.6) .. (0,-0.5); 
\draw[black,thick][-] (0.1,-0.45) .. controls (0.7,-0.35) and (0.7,0.35) .. (0.1,0.45);
\filldraw[black] (0.5,-0.79) circle (2pt) node[anchor=north]{$d$};
\filldraw[black] (0.5,0.79) circle (2pt) node[anchor=south]{$b$};
\filldraw[black] (-0.5,-0.77) circle (2pt) node[anchor=north]{$c$};
\filldraw[black] (-0.5,0.8) circle (2pt) node[anchor=south]{$a$};
\filldraw[black] (-0.2,0) circle (2pt) node[anchor=east]{$i$};
\filldraw[black] (0.55,0) circle (2pt) node[anchor=west]{$j$};
\end{tikzpicture}}; \\
\node (a) {\begin{tikzpicture}[scale=0.90]
\draw[black, dashed] (0,0) circle (40pt);
\draw[black, thick][-] (-0.7,1.2) .. controls (-0.6,0.2) and (-0.2,0) .. (0.7,-1.24);
\draw[black, thick][-] (-1.4,0.1) .. controls (-0.9,0.2) .. (-0.55,0.4);
\draw[black, thick][-] (-0.44,0.48) .. controls (0,0.8) .. (0.48,0.46); 
\draw[black, thick][-] (1.4,0.1) .. controls (0.9,0.2) .. (0.6,0.4);             
\draw[black, thick][-] (-0.7,-1.2) ..controls (-0.5,-0.8) .. (0,-0.4);     
\draw[black, thick][-] (0.1,-0.35) .. controls (0.3,-0.2) and (0.5,0) .. (0.7,1.2);                         
\filldraw[black] (-0.54,-0.9) circle (2pt) node[anchor=east]{$a$};
\filldraw[black] (0.47,-0.9) circle (2pt) node[anchor=west]{$b$};
\filldraw[black] (1.05,0.15) circle (2pt) node[anchor=north]{$c$};
\filldraw[black] (0.64,0.85) circle (2pt) node[anchor=east]{$d$};
\filldraw[black] (-1.05,0.18) circle (2pt) node[anchor=north]{$f$};
\filldraw[black] (-0.64,0.85) circle (2pt) node[anchor=west]{$e$};
\filldraw[black] (-0.25,0) circle (2pt) node[anchor=east]{$g$};
\filldraw[black] (0.4,0) circle (2pt) node[anchor=west]{$h$};
\filldraw[black] (0,0.7) circle (2pt) node[anchor=north]{$i$};
\end{tikzpicture}}; & 
\node (b) {\begin{tikzpicture}[scale=0.90]
\draw[black, dashed] (0,0) circle (40pt);
\draw[black, thick][-] (-0.7,1.2) .. controls (-0.6,0.2) and (0.6,-0.2) .. (0.7,-1.24);
\draw[black, thick][-] (-1.4,0.1) .. controls (-0.9,-0.2) .. (-0.55,-0.4);
\draw[black, thick][-] (-0.44,-0.48) .. controls (0,-0.8) .. (0.38,-0.54); 
\draw[black, thick][-] (1.4,0.1) .. controls (0.9,-0.2) .. (0.47,-0.47);             
\draw[black, thick][-] (-0.7,-1.2) ..controls (-0.6,-0.5) .. (-0.08,0);     
\draw[black, thick][-] (0.05,0.05) .. controls (0.5,0.4) .. (0.7,1.2);                         
\filldraw[black] (-0.65,-0.9) circle (2pt) node[anchor=east]{$a$};
\filldraw[black] (0.63,-0.9) circle (2pt) node[anchor=west]{$b$};
\filldraw[black] (1.05,-0.1) circle (2pt) node[anchor=south]{$c$};
\filldraw[black] (0.6,0.85) circle (2pt) node[anchor=east]{$d$};
\filldraw[black] (-1.05,-0.1) circle (2pt) node[anchor=south]{$f$};
\filldraw[black] (-0.6,0.85) circle (2pt) node[anchor=west]{$e$};
\filldraw[black] (-0.3,-0.2) circle (2pt) node[anchor=east]{$h$};
\filldraw[black] (0.2,-0.2) circle (2pt) node[anchor=west]{$g$};
\filldraw[black] (0,-0.73) circle (2pt) node[anchor=north]{$i$};
\end{tikzpicture}};\\
};
\draw [<-,blue,thick] (c.east) -- (d.west) node [above,midway] {$R_{1}$};
\draw [->,blue,thick] (d.east) -- (j.west) node [above,midway] {$R_{1}'$};
\draw [<-,blue,thick] (e.east) -- (f.west) node [above,midway] {$R_2$};
\draw [->,blue,thick] (f.east) -- (g.west) node [above,midway] {$R_{2}'$};
\draw [->,blue,thick] (a.east) -- (b.west) node [above,midway] {$R_{3}$};
\end{tikzpicture}
\caption{The triangulated Reidemeister moves (every move, denoted by an arrow, has its inverse move, obtained by reversing the arrow)}
\label{fig11}
\end{figure}

\begin{figure}[H]
\begin{tikzpicture}
\matrix [column sep=1cm]
{
\node (b) {\begin{tikzpicture}[scale=0.90]
\draw[black, dashed] (0,0) circle (40pt);
\draw[red, thick] (0,-0.16) circle (6pt);
\draw[blue, thick][-] (-1,1) .. controls (-0.1,0.1) and (-0.7,-0.2) .. (0,-0.4) 
                                           .. controls (0.7,-0.2) and (0.1,0.1) .. (1,1) ;
\draw[red, thick][-] (-1.3,-0.5) .. controls (-0.3,1.11) and (0.3,1.11) .. (1.3,-0.5) ;
\filldraw[black] (0,-0.37) circle (2pt) node[anchor=north]{$i$};
\filldraw[black] (-0.56,0.44) circle (2pt) node[anchor=south]{$a$};
\filldraw[black] (0.56,0.44) circle (2pt) node[anchor=south]{$b$};
\end{tikzpicture}}; & \node (c) {\begin{tikzpicture}[scale=0.90]
\draw[black, dashed] (0,0) circle (40pt);
\draw[blue, thick][-] (-1,1) .. controls (-0.3,0) and (0.3,0) .. (1,1) ;
\draw[red, thick][-] (-1.3,-0.5) .. controls (-0.3,0.47) and (0.3,0.47) .. (1.3,-0.5) ;
\filldraw[black] (0,0.25) circle (2pt) node[anchor=south]{$i$};
\end{tikzpicture}}; 
& \node (d) {\begin{tikzpicture}[scale=0.90]
\draw[black, dashed] (0,0) circle (40pt);
\draw[red, thick][-] (-1,1) .. controls (-0.3,0) and (0.3,0) .. (1,1) ;
\draw[blue, thick][-] (-1.3,-0.5) .. controls (-0.3,0.47) and (0.3,0.47) .. (1.3,-0.5) ;
\filldraw[black] (0,0.25) circle (2pt) node[anchor=south]{$i$};
\end{tikzpicture}}; 
& \node (e) {\begin{tikzpicture}[scale=0.90]
\draw[black, dashed] (0,0) circle (40pt);
\draw[blue, thick] (0,-0.16) circle (6pt);
\draw[red, thick][-] (-1,1) .. controls (-0.1,0.1) and (-0.7,-0.2) .. (0,-0.4) 
                                           .. controls (0.7,-0.2) and (0.1,0.1) .. (1,1) ;
\draw[blue, thick][-] (-1.3,-0.5) .. controls (-0.3,1.11) and (0.3,1.11) .. (1.3,-0.5) ;
\filldraw[black] (0,-0.37) circle (2pt) node[anchor=north]{$i$};
\filldraw[black] (-0.56,0.44) circle (2pt) node[anchor=south]{$a$};
\filldraw[black] (0.56,0.44) circle (2pt) node[anchor=south]{$b$};
\end{tikzpicture}}; \\
\node (f) {\begin{tikzpicture}[scale=0.90]
\draw[black, dashed] (0,0) circle (40pt);
\draw[blue, thick][-] (-1,1) .. controls (0,0.7) and (-0.7,-0.2) .. (0,-0.3) 
                                           .. controls (0.7,-0.2) and (0,0.7) .. (1,1);
\draw[blue, thick][-] (-1,-1) .. controls (-0.3,-0.4) and (0.3,-0.4) .. (1,-1);
\draw[red, thick][-] (-1.4,0.1) .. controls (-0.2,1) and (0.2,1) .. (1.4,0.1);
\draw[red, thick][-] (-0.5,-0.68) .. controls (-0.2,-0.7) and (-0.7,0.2) .. (0,0.3) 
                                           .. controls (0.7,0.2) and (0.2,-0.7) .. (0.5,-0.68);
\draw[red, thick][-] (-0.5,-0.68)  .. controls (-0.8,-0.7) .. (-1.35,-0.3);
\draw[red, thick][-] (0.5,-0.68)  .. controls (0.8,-0.7) .. (1.35,-0.3);
\filldraw[black] (0.5,-0.68) circle (2pt) node[anchor=north]{$d$};
\filldraw[black] (0.5,0.68) circle (2pt) node[anchor=south]{$b$};
\filldraw[black] (-0.5,-0.68) circle (2pt) node[anchor=north]{$c$};
\filldraw[black] (-0.5,0.68) circle (2pt) node[anchor=south]{$a$};
\filldraw[black] (-0.35,0) circle (2pt) node[anchor=east]{$i$};
\filldraw[black] (0.35,0) circle (2pt) node[anchor=west]{$j$};
\end{tikzpicture}}; & 
\node (g) {\begin{tikzpicture}[scale=0.90]
\draw[black, dashed] (0,0) circle (40pt);
\draw[blue, thick][-] (-1,1) .. controls (-0.15,0.5) and (-0.15,-0.5) .. (-1,-1) ;
\draw[blue, thick][-] (1,1) .. controls (0.15,0.5) and (0.15,-0.5) .. (1,-1) ;
\draw[red, thick][-] (-1.3,0.5) .. controls (-0.08,0.5) and (-0.08,-0.5) .. (-1.3,-0.5);
\draw[red, thick][-] (1.3,0.5) .. controls (0.08,0.5) and (0.08,-0.5) .. (1.3,-0.5);
\filldraw[black] (-0.35,0) circle (2pt) node[anchor=east]{$i$};
\filldraw[black] (0.35,0) circle (2pt) node[anchor=west]{$j$};
\end{tikzpicture}}; &
\node (h) {\begin{tikzpicture}[scale=0.90]
\draw[black, dashed] (0,0) circle (40pt);
\draw[red, thick][-] (-1,1) .. controls (0,0.7) and (-0.7,-0.2) .. (0,-0.3) 
                                           .. controls (0.7,-0.2) and (0,0.7) .. (1,1);
\draw[red, thick][-] (-1,-1) .. controls (-0.3,-0.4) and (0.3,-0.4) .. (1,-1);
\draw[blue, thick][-] (-1.4,0.1) .. controls (-0.2,1) and (0.2,1) .. (1.4,0.1);
\draw[blue, thick][-] (-0.5,-0.68) .. controls (-0.2,-0.7) and (-0.7,0.2) .. (0,0.3) 
                                           .. controls (0.7,0.2) and (0.2,-0.7) .. (0.5,-0.68);
\draw[blue, thick][-] (-0.5,-0.68)  .. controls (-0.8,-0.7) .. (-1.35,-0.3);
\draw[blue, thick][-] (0.5,-0.68)  .. controls (0.8,-0.7) .. (1.35,-0.3);
\filldraw[black] (0.5,-0.68) circle (2pt) node[anchor=north]{$d$};
\filldraw[black] (0.5,0.68) circle (2pt) node[anchor=south]{$b$};
\filldraw[black] (-0.5,-0.68) circle (2pt) node[anchor=north]{$c$};
\filldraw[black] (-0.5,0.68) circle (2pt) node[anchor=south]{$a$};
\filldraw[black] (-0.35,0) circle (2pt) node[anchor=east]{$i$};
\filldraw[black] (0.35,0) circle (2pt) node[anchor=west]{$j$};
\end{tikzpicture}}; & 
\node (i) {\begin{tikzpicture}[scale=0.90]
\draw[black, dashed] (0,0) circle (40pt);
\draw[red, thick][-] (-1,1) .. controls (-0.15,0.5) and (-0.15,-0.5) .. (-1,-1) ;
\draw[red, thick][-] (1,1) .. controls (0.15,0.5) and (0.15,-0.5) .. (1,-1) ;
\draw[blue, thick][-] (-1.3,0.5) .. controls (-0.08,0.5) and (-0.08,-0.5) .. (-1.3,-0.5);
\draw[blue, thick][-] (1.3,0.5) .. controls (0.08,0.5) and (0.08,-0.5) .. (1.3,-0.5);
\filldraw[black] (-0.35,0) circle (2pt) node[anchor=east]{$i$};
\filldraw[black] (0.35,0) circle (2pt) node[anchor=west]{$j$};
\end{tikzpicture}}; \\
\node (j) {\begin{tikzpicture}[scale=0.90]
\draw[black, dashed] (0,0) circle (40pt);
\draw[blue, thick][-] (-0.7,-1.24) .. controls (-0.3,-0.5) and (0.3,-0.5) .. (0.7,-1.24);
\draw[blue, thick][-] (-1.4,0.1) .. controls (-0.5,0.3) and (-0.3,-0.45) .. (0.3,-0.15)
                                                 .. controls (0.3,0.6) and (-0.45,0.1) .. (-0.7,1.2);
\draw[blue, thick][-] (1.4,0.1) .. controls (0.5,0.1) and (0.5,0.6) .. (0.7,1.2);
\draw[red, thick][-] (0.6,0.85) .. controls (0.75,1) and (0.35,0.4) .. (0,0.35)
                                                    .. controls (-0.5,-0.2) and (-0.35,-0.9) .. (-0.5,-0.93);
\draw[red, thick][-] (-1.05,0.15) .. controls (-0.5,0.1) and (-0.5,0.6) .. (-0.6,0.85);
\draw[red, thick][-] (1.05,0.15) .. controls (0.5,0.3) and (0.2,-0.3) .. (0.3,-0.15)
                                                    .. controls (0.1,-0.5) and (0.55,-1) .. (0.5,-0.93); 
\filldraw[black] (-0.5,-0.93) circle (2pt) node[anchor=east]{$a$};
\filldraw[black] (0.5,-0.93) circle (2pt) node[anchor=west]{$b$};
\filldraw[black] (1.05,0.15) circle (2pt) node[anchor=north]{$c$};
\filldraw[black] (0.6,0.85) circle (2pt) node[anchor=east]{$d$};
\filldraw[black] (-1.05,0.15) circle (2pt) node[anchor=north]{$f$};
\filldraw[black] (-0.6,0.85) circle (2pt) node[anchor=west]{$e$};
\filldraw[black] (-0.3,-0.15) circle (2pt) node[anchor=north]{$g$};
\filldraw[black] (0.3,-0.15) circle (2pt) node[anchor=north]{$h$};
\filldraw[black] (0,0.35) circle (2pt) node[anchor=south]{$i$};
\end{tikzpicture}}; & 
\node (k) {\begin{tikzpicture}[scale=0.90]
\draw[black, dashed] (0,0) circle (40pt);
\draw[blue, thick][-] (-0.7,1.2) .. controls (-0.6,0.3) and (0.6,0.3) .. (0.7,1.2);                                            
\draw[blue, thick][-] (-1.4,0.1) .. controls (-0.5,0.3) and (-0.1,-0.5) .. (-0.7,-1.24);
\draw[blue, thick][-] (1.4,0.1) .. controls (0.3,0.3) and (-0.2,0.05) .. (-0.3,0.15)
                                                .. controls (-0.1,-0.4) and (0.5,-0.9) .. (0.7,-1.24);
\draw[red, thick][-] (0.6,0.85) .. controls (0.75,1) and (0.4,0.4) .. (0,-0.35)
                                                    .. controls (-0.3,-0.7) .. (-0.5,-0.93);
\draw[red, thick][-] (-1.05,0.15) .. controls (-0.9,0.2) .. (-0.3,0.15)
                                                .. controls (-0.5,0.6) .. (-0.6,0.85);
\draw[red, thick][-] (1.05,0.15) .. controls (0.5,0.3) and (0.1,-0.5) .. (0.5,-0.93);
\filldraw[black] (-0.5,-0.93) circle (2pt) node[anchor=east]{$a$};
\filldraw[black] (0.5,-0.93) circle (2pt) node[anchor=west]{$b$};
\filldraw[black] (1.05,0.15) circle (2pt) node[anchor=north]{$c$};
\filldraw[black] (0.6,0.85) circle (2pt) node[anchor=east]{$d$};
\filldraw[black] (-1.05,0.15) circle (2pt) node[anchor=north]{$f$};
\filldraw[black] (-0.6,0.85) circle (2pt) node[anchor=west]{$e$};
\filldraw[black] (0.3,0.15) circle (2pt) node[anchor=north]{$g$};
\filldraw[black] (-0.3,0.15) circle (2pt) node[anchor=north]{$h$};
\filldraw[black] (0,-0.35) circle (2pt) node[anchor=north]{$i$};
\end{tikzpicture}};\\
};
\draw [<->,blue,thick] (f.east) -- (g.west) node [above,midway] {};
\draw [<->,blue,thick] (h.east) -- (i.west) node [above,midway] {};
\draw [<->,blue,thick] (b.east) -- (c.west) node [above,midway] {};
\draw [<->,blue,thick] (d.east) -- (e.west) node [above,midway] {};
\draw [<->,blue,thick] (j.east) -- (k.west) node [above,midway] {};
\end{tikzpicture}
\caption{The web moves}
\label{fig7}
\end{figure}

\subsection{Symmetries of a triangulated cube of smoothings} \label{subs22}

In this subsection, we present a combinatorial link invariant, based on the symmetries of triangulated smoothings. 

Let $(\mathcal{D},\tau)$ be a triangulated link diagram with $k$ crossings. Recall that for any $\mathbf{v}\in \{0,1\}^{k}$, connected components of the smoothing $\Gamma _{\mathbf{v}}$ induce a partition $\mathcal{P}_{\mathbf{v}}$ of the set $V=\{1,2,\ldots ,2k+\mu \}$. Two labels $i,j\in V$ are called $\mathbf{v}$-\emph{adjacent} if they belong to the same element of $\mathcal{P}_{\mathbf{v}}$ and if they are the labels of two adjacent vertices on the corresponding circle of $\Gamma _{\mathbf{v}}$. Note that if $i$ and $j$ are $\mathbf{v}$-adjacent, there exist two adjacent edges $e,e'\in E(\mathcal{D})$ of the graph $\mathcal{D}$ with $i=\tau (e)$ and $j=\tau (e')$. 

A permutation $\sigma \in S_{2k+\mu }$ \emph{preserves} the smoothing $\Gamma _{\mathbf{v}}$ if for every $i,j\in V$ we have $$i\textrm{ and $j$ are $\mathbf{v}$-adjacent} \Leftrightarrow \sigma (i)\textrm{ and $\sigma (j)$ are $\mathbf{v}$-adjacent}\;.$$
Thus, such permutation preserves both partition $\mathcal{P}_{\mathbf{v}}$ and the adjacency of vertices. The \emph{group of the smoothing} $\Gamma _{\mathbf{v}}$ is the subgroup $G_{\mathbf{v}}$ of $S_{2k+\mu }$ that contains all permutations preserving $\Gamma _{\mathbf{v}}$. 

\begin{example} The group of triangulated smoothing $\Gamma _{00}$ in Figure \ref{fig3} is isomorphic to the dihedral group $D_{4}$ and contains the following permutations: $$G_{00}=\{\textrm{id},\cycle{1,2},\cycle{3,4},\cycle{1,2}\cycle{3,4},\cycle{1,3}\cycle{2,4},\cycle{1,4}\cycle{2,3},\cycle{1,3,2,4},\cycle{1,4,2,3}\}\;.$$ Observe that some elements of this group correspond to symmetries of the connected components of the smoothing (e.g. $\cycle{1,2}\cycle{3,4}$), while other elements permute the connected components (e.g. $\cycle{1,4}\cycle{2,3}$). 
\end{example}

Recall that the $k$-dimensional hypercube $Q_{k}$ is a graph with the set of vertices $\{0,1\}^{k}$, whose vertices $\mathbf{v},\mathbf{v}'\in \{0,1\}^{k}$ are connected by an edge iff they differ exactly at one letter. A \emph{vertex labeling} of $Q_k$ is a function $l\colon \{0,1\}^{k}\to S$ that assigns to each vertex $\mathbf{v}\in \{0,1\}^{k}$ a label $l(\mathbf{v})$ from a set $S$.  

\begin{definition} Let $(\mathcal{D},\tau)$ be a triangulated link diagram with $k$ crossings, and let $o$ be a web orientation on its cube of smoothings. We order the components of every smoothing by the minimal vertex label in every component. The orientation $o$ of a triangulated smoothing $\Gamma _{\mathbf{v}}$ defines a permutation $\sigma _{\mathbf{v}}^{o}\in G_{\mathbf{v}}$ as follows. Let $\sigma _{\mathbf{v}}^{o}(i)$ be the label of the adjacent vertex of $i$ that follows $i$ in the oriented circle of $\Gamma _{\mathbf{v}}$, containing $i$. In case the circle contains only one vertex, we set $\sigma _{\mathbf{v}}^{o}(i)=i$. The permutation $\sigma _{\mathbf{v}}^{o}$ may be written as a product of disjoint cycles $$\sigma _{\mathbf{v}}^{o}=\prod _{j=1}^{c_{\mathbf{v}}}\lambda _{\mathbf{v},j}^{o}\;,$$ each of which acts on one component of the smoothing $\Gamma _{\mathbf{v}}$. The element $\sigma _{\mathbf{v}}^{o}\in G_{\mathbf{v}}$ will be called the \emph{permutation of the triangulated smoothing $\Gamma _{\mathbf{v}}$ with web orientation $o$}. The \emph{cube of permutation data} $\mathcal{C}(\mathcal{D})$, associated to the triangulated diagram $(\mathcal{D},\tau )$ with web orientation $o$, is the graph $Q_k$ with a vertex labeling $l\colon \{0,1\}^{k}\to S_{2k+\mu }$, given by $l(\mathbf{v})=\sigma _{\mathbf{v}}^{o}$ for every $\mathbf{v}\in \{0,1\}^{k}$. 
\end{definition} 

\begin{example} \label{ex1} Suppose $\Sigma $ carries the usual counterclockwise orientation. The triangulated diagram in Figure \ref{fig3} gives rise to the cube of permutation data with vertex labels $\{\sigma _{00}^{o}=\cycle{1,2}\cycle{3,4}, \sigma _{01}^{o}=\cycle{1,4,3,2}, \sigma _{10}^{o}=\cycle{1,2,3,4}, \sigma _{11}^{o}=\cycle{1,4}\cycle{2,3}\}$.
\begin{figure}[H]
\begin{tikzpicture}
\filldraw[black] (-2,0) circle (2pt) node[anchor=east]{$\cycle{1,2}\cycle{3,4}$};
\filldraw[black] (0,1) circle (2pt) node[anchor=south]{$\cycle{1,4,3,2}$};
\filldraw[black] (0,-1) circle (2pt) node[anchor=north]{$\cycle{1,2,3,4}$};
\filldraw[black] (2,0) circle (2pt) node[anchor=west]{$\cycle{1,4}\cycle{2,3}$};
\draw[blue, thick][-] (-2,0)--(0,1);         
\draw[blue, thick][-] (-2,0)--(0,-1);       
\draw[blue, thick][-] (0,1)--(2,0);       
\draw[blue, thick][-] (0,-1)--(2,0);                 
\end{tikzpicture}
\label{fig12}
\caption{The cube of permutation data $\mathcal{C}(\mathcal{D})$, associated with the triangulated diagram $\mathcal{D}$ in Figure \ref{fig3}.}
\end{figure}
\end{example}

\begin{remark} As a graph, the cube of permutation data is just the ordinary hypercube, while the labels of its vertices store important information about the triangulated cube of resolutions.
\end{remark}

\begin{proposition} \label{prop1} Let $c$ be the $l$-th crossing of a triangulated diagram $(\mathcal{D},\tau)$ and denote by $e_{1},e_{2},e_{3},e_{4}$ the four edges, incident at $c$, so that $e_{1}\cup e_{2}$ represents the overcrossing arc and $e_{3}\cup e_{4}$ represents the undercrossing arc. We denote by $i=\tau (e_{1})$, $j=\tau (e_2)$, $v=\tau (e_3)$ and $w=\tau (e_4)$ the labels of these edges. Let $\mathbf{v},\mathbf{v}'\in \{0,1\}^{k}$ be two words with $\mathbf{v}_{l}=0$, $\mathbf{v}_{l}'=1$ and $\mathbf{v}_{m}=\mathbf{v}_{m}'$ for every $m\in \{1,2,\ldots ,k\}\backslash \{l\}$. For any web orientation $o$ on the cube of smoothings, we have: \begin{itemize}
\item[(a)] if $\sigma _{\mathbf{v}}^{o}(i)=v$ and $\sigma _{\mathbf{v}}^{o}(j)=w$, then $\cycle{v,w}\circ \sigma _{\mathbf{v}}^{o}=\sigma _{\mathbf{v}'}^{o}$.
\item[(b)] if $\sigma _{\mathbf{v}}^{o}(v)=i$ and $\sigma _{\mathbf{v}}^{o}(w)=j$, then $\sigma _{\mathbf{v}}^{o}\circ \cycle{v,w}=\sigma _{\mathbf{v}'}^{o}$.
\end{itemize} 
\end{proposition}
\begin{proof} The smoothing $\Gamma _{\mathbf{v}'}$ is obtained from $\Gamma _{\mathbf{v}}$ by surgery along a band with vertices $i$, $w$, $j$ and $v$, see Figure \ref{fig10}. Suppose that $\sigma _{\mathbf{v}}^{o}(i)=v$ and $\sigma _{\mathbf{v}}^{o}(j)=w$. It follows that $\left (\cycle{v,w}\circ \sigma _{\mathbf{v}}^{o}\right )(i)=w$ and $\left (\cycle{v,w}\circ \sigma _{\mathbf{v}}^{o}\right )(j)=v$, while $\left (\cycle{v,w}\circ \sigma _{\mathbf{v}}^{o}\right )(m)=\sigma _{\mathbf{v}}^{o}(m)$ for every $m\in V-\{i,j\}$. Thus, $\cycle{v,w}\circ \sigma _{\mathbf{v}}^{o}=\sigma _{\mathbf{v}'}^{o}$. 
\begin{figure}[H]
\begin{tikzpicture}[scale=0.70]
\draw[black, dashed] (-0.7,0) arc (0:200:1.4); 
\draw[blue, thick] (-0.7,0) arc (0:38:1.4);
\draw[blue, thick] (0.7,0) arc (180:142:1.4);
\draw[black, dashed] (3.4,-0.5) arc (-20:180:1.4); 
\filldraw[black] (-0.67,0) circle (2pt) node[anchor=east]{$i$};
\filldraw[black] (-0.95,0.8) circle (2pt) node[anchor=south]{$v$};
\filldraw[black] (0.675,0) circle (2pt) node[anchor=west]{$w$};
\filldraw[black] (0.95,0.8) circle (2pt) node[anchor=south]{$j$};
\draw[-,red,thick] (-0.67,0) .. controls (-0.4,0.3) and (0.4,0.3) .. (0.67,0) node [above,midway] {};
\draw[-,red,thick] (-0.95,0.8) .. controls (-0.7,1.3) and (0.7,1.3) .. (0.95,0.8) node [above,midway] {};
\end{tikzpicture}
\caption{Proof of Proposition \ref{prop1}.}
\label{fig10}
\end{figure}
In case (b), changing the web orientation would result in the case (a) situation. By taking inverses of both sides of the equation $\cycle{v,w}\circ \sigma _{\mathbf{v}}^{-o}=\sigma _{\mathbf{v}'}^{-o}$, we obtain $\sigma _{\mathbf{v}}^{o}\circ \cycle{v,w}=\sigma _{\mathbf{v}'}^{o}$.
\end{proof}

\begin{corollary} \label{cor1} Let $\mathbf{v},\mathbf{v}'\in \{0,1\}^{k}$ be two words denoting vertices in the cube of smoothings with a web orientation $o$. 
\begin{enumerate}
\item If $r_{\mathbf{v}}=r_{\mathbf{v'}}$, then the permutations of the triangulated smoothings $\sigma _{\mathbf{v}}^{o}$ and $\sigma _{\mathbf{v}'}^{o}$ have the same parity. 
\item  If $r_{\mathbf{v'}}=r_{\mathbf{v}}+1$, then the permutations of the triangulated smoothings $\sigma _{\mathbf{v}}^{o}$ and $\sigma _{\mathbf{v}'}^{o}$ have different parity. 
\end{enumerate}
\end{corollary}
\begin{proof}
\begin{enumerate}
\item We use induction on the number $r_{\mathbf{v}}$. If $r_{\mathbf{v}}=0$, the statement is trivial since there is only one such vertex $\mathbf{v}=0^{k}$. Now let $n\in \mathbb{N}$ and suppose that the statement is true for all pairs $\mathbf{v}, \mathbf{v}'\in \{0,1\}^{k}$ with $r_{\mathbf{v}}=r_{\mathbf{v}'}<n$. Choose any two words $\mathbf{w}, \mathbf{w}'\in \{0,1\}^{k}$ with $r_{\mathbf{w}}=r_{\mathbf{w}'}=n$ and denote $l=\min \{j\in \{1,2,\ldots k\}\,|\, \mathbf{w}_{j}=1\}$, $l'=\min \{j\in \{1,2,\ldots k\}\,|\, \mathbf{w}_{j}'=1\}$. Let $\mathbf{v}$ (resp. $\mathbf{v}'$) be the word in $\{0,1\}^{k}$ that differs from $\mathbf{w}$ (resp. $\mathbf{w}'$) only at index $l$ (resp. $l'$). Since $r_{\mathbf{v}}=r_{\mathbf{v}'}=n-1$, the permutations $\sigma _{\mathbf{v}}^{o}$ and $\sigma _{\mathbf{v}'}^{o}$ have the same parity by the induction hypothesis. It follows by Proposition \ref{prop1} that also $\sigma _{\mathbf{w}}^{o}$ and $\sigma _{\mathbf{w}'}^{o}$ have the same parity. 
\item Suppose that $r_{\mathbf{v'}}=r_{\mathbf{v}}+1$. We may choose a word $\mathbf{w}\in \{0,1\}^{k}$ with $r_{\mathbf{w}}=r_{\mathbf{v}}$, so that the words $\mathbf{w}$ and $\mathbf{v}'$ differ exactly at one index; denote that index by $l$. Thus $\mathbf{w}_{l}=0$ and $\mathbf{v}'_{l}=1$. By Proposition \ref{prop1}, the permutation $\sigma _{\mathbf{v}'}^{o}$ has different parity than the permutation $\sigma _{\mathbf{w}}^{o}$. Since the parity of permutations $\sigma _{\mathbf{w}}^{o}$ and $\sigma _{\mathbf{v}}^{o}$ is equal by (1), the conclusion follows.  
\end{enumerate}
\end{proof}

\begin{lemma} \label{lemma2} Let $\mathcal{C}(\mathcal{D})$ be the cube of permutation data, associated to a triangulated link diagram $(\mathcal{D},\tau)$ with web orientation $o$ and vertex labeling $l\colon \{0,1\}^{k}\to S_n$. Suppose that a triangulated link diagram $(\mathcal{D}',\tau ')$ is obtained from $(\mathcal{D},\tau )$ by the Reidemeister move $R_1$ (or $R_{1}'$) on the edge $e\in E(\mathcal{D})$ with $\tau (e)=i$, where $\tau '(e)=\tau (e)$ for $e\in E(\mathcal{D})$ and $\tau '(E(\mathcal{D}')\backslash E(\mathcal{D}))\subseteq\{n+1,n+2\}$. Then $\mathcal{C}(\mathcal{D}')$ is given by a vertex labeling \begin{itemize}
\item $l'\colon \{0,1\}^{k+1}\to S_{n+2}$, $l'(\mathbf{w})=\sigma '_{\mathbf{w}}$ with $$\left \{\sigma _{\mathbf{v}\times \{0\}}',\sigma _{\mathbf{v}\times \{1\}}'\right \}=\left \{ \cycle{\sigma _{\mathbf{v}}(i),n+2}\cycle{i,n+1}\circ \sigma _{\mathbf{v}},\cycle{i,n+1,n+2,\sigma _{\mathbf{v}}(i)}\circ \sigma _{\mathbf{v}}\right \}$$ for every $\mathbf{v}\in \{0,1\}^{k}$ if $|N(i)|>0$ in the diagram $\mathcal{D}$,
\item $l'\colon \{0,1\}^{k+1}\to S_{n+1}$, $l'(\mathbf{w})=\sigma '_{\mathbf{w}}$ with $\left \{\sigma _{\mathbf{v}\times \{0\}}',\sigma _{\mathbf{v}\times \{1\}}'\right \}=\left \{ \sigma _{\mathbf{v}},\cycle{i,n+1}\sigma _{\mathbf{v}}\right \}$ for every $\mathbf{v}\in \{0,1\}^{k}$ if $|N(i)|=0$ in the diagram $\mathcal{D}$. 
\end{itemize}
\end{lemma}
\begin{proof} During the Reidemeister move, the newly formed crossing splits the edge with label $i$ into three parts if $|N(i)|>0$, or into two parts if $|N(i)|=0$. In the first case, the newly formed edges are labeled by $n+1$ and $n+2$, while in the second case, we only have one new edge labeled by $n+1$. In the first case, every vertex label $\sigma _{\mathbf{v}}\in S_{n}$ of $Q_{k}$ in $\mathcal{C}(\mathcal{D})$ gives two new vertex labels $\sigma _{\mathbf{w}}',\sigma _{\mathbf{w}'}'\in S_{n+2}$ of $Q_{k+1}$ in $\mathcal{C}(\mathcal{D}')$ with $\mathbf{w},\mathbf{w}'\in \mathbf{v}\times \{0,1\}$, namely $\sigma _{\mathbf{w}}'=\cycle{\sigma _{\mathbf{v}}(i),n+2}\cycle{i,n+1}\circ \sigma _{\mathbf{v}}$ and $\sigma _{\mathbf{w}'}'=\cycle{i,n+1,n+2,\sigma _{\mathbf{v}}(i)}\circ \sigma _{\mathbf{v}}$, see Figure \ref{fig6}. In the second case, every vertex label $\sigma _{\mathbf{v}}\in S_{n}$ of $Q_{k}$ in $\mathcal{C}(\mathcal{D})$ gives two new vertex labels $\sigma _{\mathbf{w}}',\sigma _{\mathbf{w}'}'\in S_{n+1}$ of $Q_{k+1}$ in $\mathcal{C}(\mathcal{D}')$ with $\mathbf{w},\mathbf{w}'\in \mathbf{v}\times \{0,1\}$, namely $\sigma _{\mathbf{w}}'=\sigma _{\mathbf{v}}$ and $\sigma _{\mathbf{w}'}'=\cycle{i,n+1}\sigma _{\mathbf{v}}$.
\end{proof}
\begin{figure}[H]
\begin{tikzpicture}
\matrix [column sep=1cm]
{
\node (c) {\begin{tikzpicture}[scale=0.80]
\draw[black, dashed] (0,0) circle (40pt);
\draw[black, thick][-] (-1,1) .. controls (-0.3,0) and (0.3,0) .. (1,1) ;
\filldraw[black] (0,0.27) circle (2pt) node[anchor=north]{\footnotesize $i$};
\end{tikzpicture}}; &
\node (a) {\begin{tikzpicture}[scale=0.80]
\draw[black, dashed] (0,0) circle (40pt);
\draw[black, thick] (0,-0.3) circle (10pt);
\draw[black, thick][-] (-1,1) .. controls (-0.3,0) and (0.3,0) .. (1,1) ;
\filldraw[black] (0,-0.65) circle (2pt) node[anchor=north]{\footnotesize $i$};
\filldraw[black] (0.5,0.44) circle (2pt) node[anchor=south]{\footnotesize $n+2$};
\filldraw[black] (-0.5,0.44) circle (2pt) node[anchor=south]{\footnotesize $n+1$};
\end{tikzpicture}}; & \node (b) {\begin{tikzpicture}[scale=0.80]
\draw[black, dashed] (0,0) circle (40pt);
\draw[black, thick][-] (-1,1) .. controls (-0.1,0.1) and (-0.7,-0.2) .. (0,-0.4) 
                                           .. controls (0.7,-0.2) and (0.1,0.1) .. (1,1) ;
\filldraw[black] (0,-0.4) circle (2pt) node[anchor=north]{\footnotesize $i$};
\filldraw[black] (-0.55,0.44) circle (2pt) node[anchor=south]{\footnotesize $n+1$};
\filldraw[black] (0.55,0.44) circle (2pt) node[anchor=south]{\footnotesize $n+2$};
\end{tikzpicture}}; \\
};
\draw [-,blue,thick] (a.east) -- (b.west) node [above,midway] {};
\draw [-,red,thick] (-2.2,-0.1) .. controls (-2,0.5) and (-1.6,-0.5) .. (-1.3,0.1) node [above,midway] {};
\end{tikzpicture}
\caption{Triangulated smoothings during the first Reidemeister move}
\label{fig6}
\end{figure}

\begin{lemma} \label{lemma3} Let $\mathcal{C}(\mathcal{D})$ be the cube of permutation data, associated to a triangulated link diagram $(\mathcal{D},\tau)$ with web orientation $o$ and vertex labeling $l\colon \{0,1\}^{k}\to S_n$. Suppose that there exists labels $i\neq m,p\in \{1,2,\ldots ,n\}$, $j\in \{1,2,\ldots ,k\}$ and $\epsilon \in \{0,1\}$ such that $\sigma _{\mathbf{v}}(i)=\begin{cases} i & \textrm{if $\mathbf{v}_{j}=\epsilon $}\\
m & \textrm{if $\mathbf{v}_{j}\neq \epsilon $}\end{cases}$ and $\sigma _{\mathbf{v}}(p)=\begin{cases}m & \textrm{if $\mathbf{v}_{j}=\epsilon $}\\  i & \textrm{if $\mathbf{v}_{j}\neq \epsilon $}\end{cases}$. Let a triangulated link diagram $(\mathcal{D}',\tau ')$ be obtained from $(\mathcal{D},\tau )$ by the inverse of Reidemeister move $R_{1}$ (or $R_{1}'$), removing the edge $e\in E(\mathcal{D})$ with $\tau (e)=i$. 
\begin{itemize}
\item If $p\neq m$, $\mathcal{C}(\mathcal{D}')$ is given by a vertex labeling $l'\colon \{0,1\}^{k-1}\to Sym\left (\{1,2,\ldots n\}\backslash \{i,p\}\right )\cong S_{n-2}$, $l'(\mathbf{v})=\sigma '_{\mathbf{v}}$ with $\sigma _{\mathbf{v}\backslash \{v_j\}}'=\cycle{m,p}\sigma _{\mathbf{v}}$ for $\mathbf{v}\in \{\mathbf{w}\in \{0,1\}^{k}\,|\,\mathbf{w}_{j}=\epsilon \}$.
\item  If $p=m$, $\mathcal{C}(\mathcal{D}')$ is given by a vertex labeling $l'\colon \{0,1\}^{k-1}\to Sym\left (\{1,2,\ldots n\}\backslash \{i,p\}\right )\cong S_{n-2}$, $l'(\mathbf{v})=\sigma '_{\mathbf{v}}$ with $\sigma _{\mathbf{v}\backslash \{v_j\}}'=\sigma _{\mathbf{v}}$ for $\mathbf{v}\in \{\mathbf{w}\in \{0,1\}^{k}\,|\,\mathbf{w}_{j}=\epsilon \}$.
\end{itemize}
\end{lemma}
\begin{proof} Since $i$ is a fixed point of permutation labels of all vertices $\mathbf{v}$ of $Q_k$ with $\mathbf{v}_{j}=\epsilon $, the edge with label $i$ is incident with only one crossing of the diagram $\mathcal{D}$, thus this crossing can be removed by the inverse of the first Reidemeister move. \begin{itemize}
\item If $p\neq m$, then the vertex labeling $\mathcal{C}(\mathcal{D}')$ of $Q_{k-1}$ is obtained from the part of the vertex labeling of $\mathcal{C}(\mathcal{D})$ with $\mathbf{v}_{j}=\epsilon $ by removing the labels $i$ and $p$ from all permutation labels. Since all these permutation labels satisfy $\sigma _{\mathbf{v}}(i)=i$ and $\sigma _{\mathbf{v}}(p)=m$, composing by the transposition $\cycle {p,m}$ yields the desired result.  
\item If $p=m$, both labels $i$ and $p$ are fixed points of every permutation $\sigma _{\mathbf{v}}$ with $\mathbf{v}_{j}=\epsilon $. Thus, the part of the vertex labeling of $\mathcal{C}(\mathcal{D})$ with $\mathbf{v}_{j}=\epsilon $ yields a vertex labeling of $\mathcal{C}(\mathcal{D}')$. 
\end{itemize}
\end{proof}

\begin{lemma} \label{lemma4} Let $\mathcal{C}(\mathcal{D})$ be the cube of permutation data, associated to a triangulated link diagram $(\mathcal{D},\tau)$ with web orientation $o$ and vertex labeling $l\colon \{0,1\}^{k}\to S_n$. Suppose that a triangulated link diagram $(\mathcal{D}',\tau ')$ is obtained from $(\mathcal{D},\tau )$ by the Reidemeister move $R_2$ (or $R_{2}'$), during which the edge with label $j$ passes over (or under) the edge with label $i$. We denote $\tau '(e)=\tau (e)$ for $e\in E(\mathcal{D})$ and $\tau '(E(\mathcal{D}')\backslash E(\mathcal{D}))\subseteq\{a,b,c,d\}$. Then for some web orientation $o'$ of the diagram $(\mathcal{D}',\tau ')$, the cube of permutation data $\mathcal{C}(\mathcal{D}')$ is given by a vertex labeling $l'$ of $Q_{k+2}$, $l'(\mathbf{w})=\sigma _{\mathbf{w}}'$. \begin{itemize}
\item[(a)] If $|N(i)|>0$, $|N(j)|>0$ and $N(i)\cap N(j)=\emptyset $ in the diagram $\mathcal{D}$, then  \begin{xalignat*}{1}
& \left \{\sigma _{\mathbf{v}\times \{(0,0)\}}',\sigma _{\mathbf{v}\times \{(1,1)\}}'\right \}=\left \{\cycle{j,b}\cycle{\sigma _{\mathbf{v}}(i),a}\cycle{\sigma _{\mathbf{v}}(j),i,c,d}\sigma _{\mathbf{v}},\cycle{i,c}\cycle{\sigma _{\mathbf{v}}(j),d}\cycle{\sigma _{\mathbf{v}}(i),j,b,a}\sigma _{\mathbf{v}}\right \}\\
& \left \{\sigma _{\mathbf{v}\times \{(0,1)\}}',\sigma _{\mathbf{v}\times \{(1,0)\}}'\right \}=\left \{\cycle{j,b}\cycle{\sigma _{\mathbf{v}}(j),d}\cycle{\sigma _{\mathbf{v}}(i),a}\cycle{i,c}\sigma _{\mathbf{v}},\cycle{\sigma _{\mathbf{v}}(i),,j,b,a}\cycle{\sigma _{\mathbf{v}}(j),i,c,d}\sigma _{\mathbf{v}}\right \}
\end{xalignat*} 

\item[(b)] If $|N(i)|>0$ and $|N(j)|=0$ in the diagram $\mathcal{D}$, then $l'\colon \{0,1\}^{k+2}\to S_{n+3}$ and \begin{xalignat*}{1}
& \left \{\sigma _{\mathbf{v}\times \{(0,0)\}}',\sigma _{\mathbf{v}\times \{(1,1)\}}'\right \}=\left \{\cycle{\sigma _{\mathbf{v}}(i),a}\cycle{i,c,b,j}\sigma _{\mathbf{v}}, \cycle{\sigma _{\mathbf{v}}(i),j,b,a}\cycle{i,c}\sigma _{\mathbf{v}}\right \}\\
& \left \{\sigma _{\mathbf{v}\times \{(0,1)\}}',\sigma _{\mathbf{v}\times \{(1,0)\}}'\right \}=\left \{\cycle{j,b}\cycle{\sigma _{\mathbf{v}}(i),a}\cycle{i,c}\sigma _{\mathbf{v}}, \cycle{i,c,b,a,\sigma _{\mathbf{v}}(i),j}\sigma _{\mathbf{v}}\right \}
\end{xalignat*}

\item[(c)] If $|N(i)|=0$ and $|N(j)|>0$ in the diagram $\mathcal{D}$, then $l'\colon \{0,1\}^{k+2}\to S_{n+3}$ and \begin{xalignat*}{1}
& \left \{\sigma _{\mathbf{v}\times \{(0,0)\}}',\sigma _{\mathbf{v}\times \{(1,1)\}}'\right \}=\left \{\cycle{j,b}\cycle{i,a,d,\sigma _{\mathbf{v}}(j)}\sigma _{\mathbf{v}}, \cycle{i,j,b,a}\cycle{\sigma _{\mathbf{v}}(j),d}\sigma _{\mathbf{v}}\right \}\\
& \left \{\sigma _{\mathbf{v}\times \{(0,1)\}}',\sigma _{\mathbf{v}\times \{(1,0)\}}'\right \}=\left \{\cycle{j,b}\cycle{\sigma _{\mathbf{v}}(j),d}\cycle{i,a}\sigma _{\mathbf{v}}, \cycle{i,j,b,a,d,\sigma _{\mathbf{v}}(j)}\sigma _{\mathbf{v}}\right \}
\end{xalignat*} 

\item[(d)] If $|N(i)|=|N(j)|=0$ in the diagram $\mathcal{D}$, then $l'\colon \{0,1\}^{k+2}\to S_{n+2}$ and \begin{xalignat*}{1}
& \left \{\sigma _{\mathbf{v}\times \{(0,0)\}}',\sigma _{\mathbf{v}\times \{(1,1)\}}'\right \}=\{\cycle{j,i,a,b}\sigma _{\mathbf{v}},\cycle{i,j,b,a}\sigma _{\mathbf{v}}\}\\
& \left \{\sigma _{\mathbf{v}\times \{(0,1)\}}',\sigma _{\mathbf{v}\times \{(1,0)\}}'\right \}=\{\cycle{a,i}\cycle{j,b}\sigma _{\mathbf{v}},\cycle{a,b}\cycle{i,j}\sigma _{\mathbf{v}}\}
\end{xalignat*}  

%\item $l'\colon \{0,1\}^{k+2}\to S_{n+2}$, $l'(\mathbf{w})=\sigma _{\mathbf{w}}'$ with \begin{xalignat*}{1}
%& \left \{\sigma _{\mathbf{v}\times \{(0,0)\}}',\sigma _{\mathbf{v}\times \{(1,1)\}}'\right \}=\{\cycle{i,a,j},\cycle{i,j,c}\}\\
%& \left \{\sigma _{\mathbf{v}\times \{(0,1)\}}',\sigma _{\mathbf{v}\times \{(1,0)\}}'\right \}=\{\cycle{i,a,j,c},\cycle{i,j}\}
%\end{xalignat*}  if $|N(i)\cap N(j)|=2$. 

%\item \begin{xalignat*}{1}
%& \left \{\sigma _{\mathbf{v}\times \{(0,0)\}}',\sigma _{\mathbf{v}\times \{(1,1)\}}'\right \}=\left \{\cycle{j,a}\cycle{\sigma _{\mathbf{v}}(j),i,c,d}\sigma _{\mathbf{v}}, \ldots \sigma _{\mathbf{v}}\right \}\\
%& \left \{\sigma _{\mathbf{v}\times \{(0,1)\}}',\sigma _{\mathbf{v}\times \{(1,0)\}}'\right \}=\left \{\right \}
%\end{xalignat*} if $|N(i)\cap N(j)|=1$. 

\end{itemize} In the remaining possible cases when $|N(i)\cap N(j)|>0$, the new pair of crossings produced by the Reidemeister move $R_2$ (or $R_{2}'$) may be removed by a $R_{1}^{-1}$ and a $R_{1}'^{-1}$ move, which were dealt with in Lemma \ref{lemma3}. 
\end{lemma}
\begin{figure}[h!]
\begin{tikzpicture}
\matrix [column sep=1cm]
{
\node {}; & \node (01) {\begin{tikzpicture}[scale=0.80]
\draw[black, dashed] (0,0) circle (40pt);
\draw[black, thick][-] (-1,1) .. controls (0,0.5) and (-0.2,0.2) .. (-0.35,0)
                                           .. controls (-0.2,-0.2) and (0,-0.5) .. (-1,-1);
\draw[black, thick][-] (1,1) .. controls (0,0.5) and (0.2,0.2) .. (0.35,0) 
                                           .. controls (0.2,-0.2) and (0,-0.5) .. (1,-1);
\filldraw[black] (0.5,-0.68) circle (2pt) node[anchor=north]{$d$};
\filldraw[black] (0.5,0.68) circle (2pt) node[anchor=south]{$b$};
\filldraw[black] (-0.5,-0.68) circle (2pt) node[anchor=north]{$c$};
\filldraw[black] (-0.5,0.68) circle (2pt) node[anchor=south]{$a$};
\filldraw[black] (-0.35,0) circle (2pt) node[anchor=east]{$i$};
\filldraw[black] (0.35,0) circle (2pt) node[anchor=west]{$j$};
\end{tikzpicture}}; & \node {};\\ 
\node (00) {\begin{tikzpicture}[scale=0.80]
\draw[black, dashed] (0,0) circle (40pt);
\draw[black, thick][-] (-1,1) .. controls (0,0.7) and (-0.7,-0.2) .. (0,-0.3) 
                                           .. controls (0.7,-0.2) and (0,0.7) .. (1,1);
\draw[black, thick][-] (-1,-1) .. controls (-0.3,-0.4) and (0.3,-0.4) .. (1,-1);
\filldraw[black] (0.5,-0.68) circle (2pt) node[anchor=north]{$d$};
\filldraw[black] (0.5,0.68) circle (2pt) node[anchor=south]{$b$};
\filldraw[black] (-0.5,-0.68) circle (2pt) node[anchor=north]{$c$};
\filldraw[black] (-0.5,0.68) circle (2pt) node[anchor=south]{$a$};
\filldraw[black] (-0.35,0) circle (2pt) node[anchor=east]{$i$};
\filldraw[black] (0.35,0) circle (2pt) node[anchor=west]{$j$};
\end{tikzpicture}}; & \node (10) {\begin{tikzpicture}[scale=0.80]
\draw[black, dashed] (0,0) circle (40pt);
\draw[black, thick] (0,0) circle (10pt);
\draw[black, thick][-] (-1,1) .. controls (-0.3,0.4) and (0.3,0.4) .. (1,1) ;
\draw[black, thick][-] (-1,-1) .. controls (-0.3,-0.4) and (0.3,-0.4) .. (1,-1) ;
\filldraw[black] (0.5,-0.68) circle (2pt) node[anchor=north]{$d$};
\filldraw[black] (0.5,0.68) circle (2pt) node[anchor=south]{$b$};
\filldraw[black] (-0.5,-0.68) circle (2pt) node[anchor=north]{$c$};
\filldraw[black] (-0.5,0.68) circle (2pt) node[anchor=south]{$a$};
\filldraw[black] (-0.35,0) circle (2pt) node[anchor=east]{$i$};
\filldraw[black] (0.35,0) circle (2pt) node[anchor=west]{$j$};
\end{tikzpicture}}; & 
\node (11) {\begin{tikzpicture}[scale=0.80]
\draw[black, dashed] (0,0) circle (40pt);
\draw[black, thick][-] (-1,1) .. controls (-0.3,0.4) and (0.3,0.4) .. (1,1) ;
\draw[black, thick][-] (-1,-1) .. controls (0,-0.7) and (-0.7,0.2) .. (0,0.3) 
                                           .. controls (0.7,0.2) and (0,-0.7) .. (1,-1) ;
\filldraw[black] (0.5,-0.68) circle (2pt) node[anchor=north]{$d$};
\filldraw[black] (0.5,0.68) circle (2pt) node[anchor=south]{$b$};
\filldraw[black] (-0.5,-0.68) circle (2pt) node[anchor=north]{$c$};
\filldraw[black] (-0.5,0.68) circle (2pt) node[anchor=south]{$a$};
\filldraw[black] (-0.35,0) circle (2pt) node[anchor=east]{$i$};
\filldraw[black] (0.35,0) circle (2pt) node[anchor=west]{$j$};
\end{tikzpicture}}; &
\node (a) {\begin{tikzpicture}[scale=0.80]
\draw[black, dashed] (0,0) circle (40pt);
\draw[black, thick][-] (-1,1) .. controls (-0.15,0.5) and (-0.15,-0.5) .. (-1,-1) ;
\draw[black, thick][-] (1,1) .. controls (0.15,0.5) and (0.15,-0.5) .. (1,-1) ;
\filldraw[black] (-0.35,0) circle (2pt) node[anchor=east]{$j$};
\filldraw[black] (0.35,0) circle (2pt) node[anchor=west]{$i$};
\end{tikzpicture}}; \\
}; 
\draw [->,blue,thick] (00.east) -- (01.west) node [above,midway] {};
\draw [->,blue,thick] (00.east) -- (10.west) node [above,midway] {};
\draw [->,blue,thick] (01.east) -- (11.west) node [above,midway] {};
\draw [->,blue,thick] (10.east) -- (11.west) node [above,midway] {};
\draw [-,red,thick] (3.2,-1.6) .. controls (3.5,-1) and (3.6,-2) .. (3.9,-1.4) node [above,midway] {};
\end{tikzpicture}
\caption{Triangulated smoothings during the second Reidemeister move}
\label{fig8}
\end{figure}

\begin{proof}  The initial diagram $\mathcal{D}$ contains $k$ crossings and $n$ labeled vertices. 
\begin{itemize}
\item[(a)] In case (a), the new diagram $\mathcal{D}'$ obtains an additional pair of crossings and four additional vertices with labels $a=n+1$, $b=n+2$, $c=n+3$ and $d=n+4$. The resulting local smoothings are shown in Figure \ref{fig8}, where switching between two types of move ($R_2$ or $R_{2}'$) results in an exchange between the $00$-smoothing and the $11$-smoothing. Looking at Figure \ref{fig8}, we may compute the following relationships for some web orientation $o'$ of the triangulated diagram $(\mathcal{D}',\tau ')$ and any $\mathbf{v}\in \{0,1\}^{k}$: 
\begin{xalignat*}{1}
& \cycle{j,b}\cycle{\sigma _{\mathbf{v}}(i),a}\cycle{i,c,d,\sigma _{\mathbf{v}}(j)}\sigma _{\mathbf{v}}=\\
& =\cycle{\ldots ,\sigma _{\mathbf{v}}^{-1}(j),b,j,i,a,\sigma _{\mathbf{v}}(i),\ldots }\cycle{\ldots ,\sigma _{\mathbf{v}}^{-1}(i),c,d,\sigma _{\mathbf{v}}(j),\ldots }=\sigma _{\mathbf{v}\times \{(0,0)\}}\\
& \cycle{j,b}\cycle{\sigma _{\mathbf{v}}(j),d}\cycle{\sigma _{\mathbf{v}}(i),a}\cycle{i,c}\sigma _{\mathbf{v}}=\\
& =\cycle{\ldots ,\sigma _{\mathbf{v}}^{-1}(j),b,j,d,\sigma _{\mathbf{v}}(j),\ldots }\cycle{\ldots ,\sigma _{\mathbf{v}}^{-1}(i),c,i,a,\sigma _{\mathbf{v}}(i),\ldots }=\sigma _{\mathbf{v}\times \{(0,1)\}}\\
& \cycle{\sigma _{\mathbf{v}}(i),j,b,a}\cycle{\sigma _{\mathbf{v}}(j),i,c,d}\sigma _{\mathbf{v}}=\\
& =\cycle{\ldots ,\sigma _{\mathbf{v}}^{-1}(i),c,d,\sigma _{\mathbf{v}}(j),\ldots }\cycle{\ldots ,\sigma _{\mathbf{v}}^{-1}(j),b,a,\sigma _{\mathbf{v}}(i),\ldots }\cycle{i,j}=\sigma _{\mathbf{v}\times \{(1,0)\}}\\
& \cycle{\sigma _{\mathbf{v}}(i),j,b,a}\cycle{\sigma _{\mathbf{v}}(j),d}\cycle{i,c}\sigma _{\mathbf{v}}=\\
& =\cycle{\ldots ,\sigma _{\mathbf{v}}^{-1}(i),c,i,j,d,\sigma _{\mathbf{v}}(j),\ldots }\cycle{\ldots ,\sigma _{\mathbf{v}}^{-1}(j),b,a,\sigma _{\mathbf{v}}(i),\ldots }=\sigma _{\mathbf{v}\times \{(1,1)\}}
\end{xalignat*}
\item[(b)] In case (b), the vertex with label $j$ is the only vertex of an empty component of $\mathcal{D}$ that contains no other label. Therefore, we need to remove the label $\sigma _{\mathbf{v}}(j)$ and identify $b=d$ in the expression (a) to obtain a similar relationship in (b).   
\item[(c)] In case (c), the vertex with label $i$ is the only vertex of an empty component of $\mathcal{D}$ that contains no other label. Therefore, we need to remove the label $\sigma _{\mathbf{v}}(i)$ and identify $a=c$ in the expression (a) to obtain a similar relationship in (c). 
\item[(d)] In case (d), $i$ and $j$ are the labels of only vertices of two empty components of $\mathcal{D}$. During the Reidemeister move $R_2$ (or $R_{2}'$), we obtain two new vertices $a$ and $b$. Observing the corresponding local smoothings yields the stated result.  
\end{itemize}
\end{proof}

The cube of permutation data $\mathcal{C}(\mathcal{D})$ depends on the ordering of crossings in the link diagram $\mathcal{D}$. A permutation of crossings changes the labeling of vertices of $\mathcal{C}(\mathcal{D})$, even though the set of vertices stays the same. A bijection $f\colon \{1,2,\ldots ,k\}\to \{1,2,\ldots ,k\}$ induces a function $\widehat{f}\colon \{0,1\}^{k}\to \{0,1\}^{k}$, defined by $\widehat{f}(\mathbf{v})_{i}=\mathbf{v}_{f(i)}$ for any $\mathbf{v}\in \{0,1\}^{k}$ and $i\in \{1,2,\ldots ,k\}$. Thus, $f$ permutes the labeling $l(\mathbf{v})=\sigma _{\mathbf{v}}$ into $l(\widehat{f}(\mathbf{v}))=\sigma _{\widehat{f}(\mathbf{v})}$.     

\begin{definition} Let $l\colon \{0,1\}^{k}\to S_{n}$ be a vertex labeling of the $k$-dimensional hypercube $Q_k$ with elements of the permutation group $S_{n}$ and denote by $\sigma _{\mathbf{v}}=l(\mathbf{v})$ the label of the vertex $\mathbf{v}$. Let $l'\colon \{0,1\}^{k'}\to S_{n'}$ be a vertex labeling of $Q_{k'}$ with elements of the permutation group $S_{n'}$ and denote $\sigma _{\mathbf{v}'}=l'(\mathbf{v})$. The vertex labeling $l'$ is called a  \emph{transformation} of the vertex labeling $l$ if one of the following holds: \begin{itemize}
\item[(a)] $k=k'$ and $\sigma _{\mathbf{v}}'=\rho \circ \sigma _{\widehat{f}(\mathbf{v})}\circ \rho ^{-1}$ for every $\mathbf{v}\in \{0,1\}^{k}$ and some permutations $f\in S_{k}$ and $\rho \in S_{n}$. 
\item[(b)] $k=k'$ and $\sigma _{\mathbf{v}}'=\zeta _{\mathbf{v}}\circ \sigma _{\mathbf{v}}\circ \zeta _{\mathbf{v}}^{-1}$, where $\zeta _{\mathbf{v}}\in G_{\mathbf{v}}$ for every $\mathbf{v}\in \{0,1\}^{k}$. 
\item[(c)] $k'=k+1$, $n'=n+2$ and for some label $i\in \{1,2,\ldots ,n\}$ with $|N(i)|>0$ we have
$$\left \{\sigma _{\mathbf{v}\times \{0\}}',\sigma _{\mathbf{v}\times \{1\}}'\right \}=\left \{ \cycle{\sigma _{\mathbf{v}}(i),n+2}\cycle{i,n+1}\circ \sigma _{\mathbf{v}},\cycle{i,n+1,n+2,\sigma _{\mathbf{v}}(i)}\circ \sigma _{\mathbf{v}}\right \}$$ 
for every $\mathbf{v}\in \{0,1\}^{k}$.  
\item[(d)] $k'=k+1$, $n'=n+1$ and for some label $i\in \{1,2,\ldots ,n\}$ with $|N(i)|=0$ we have
$$\left \{\sigma _{\mathbf{v}\times \{0\}}',\sigma _{\mathbf{v}\times \{1\}}'\right \}=\left \{ \sigma _{\mathbf{v}},\cycle{i,n+1}\sigma _{\mathbf{v}}\right \}$$ 
for every $\mathbf{v}\in \{0,1\}^{k}$.  
\item[(e)] $k'=k-1$, $n'=n-2$ and there exists a label $i\in \{1,2,\ldots ,n\}\backslash \{n,n-1\}$, $j\in \{1,2,\ldots ,k\}$ and $\epsilon \in \{0,1\}$ such that $\sigma _{\mathbf{v}}(n)=\begin{cases} n & \textrm{if $\mathbf{v}_{j}=\epsilon $}\\
i & \textrm{if $\mathbf{v}_{j}\neq \epsilon $}\end{cases}$, $\sigma _{\mathbf{v}}(n-1)=\begin{cases}i & \textrm{if $\mathbf{v}_{j}=\epsilon $}\\  n & \textrm{if $\mathbf{v}_{j}\neq \epsilon $}\end{cases}$ and $N(n)=\{n-1,i\}$; then $$\sigma _{(\mathbf{v}_{1},\ldots ,\widehat{\mathbf{v}_{j}},\ldots ,\mathbf{v}_{k})}'=\left (\cycle{n-1,i}\sigma _{\mathbf{v}}\right )|_{\{1,2,\ldots ,n-2\}} \textrm{ for $\mathbf{v}\in \{\mathbf{w}\in \{0,1\}^{k}\,|\,\mathbf{w}_{j}=\epsilon \}$}\;.$$ 
\item[(f)] $k'=k-1$, $n'=n-1$ and there exists a label $i\in \{1,2,\ldots ,n\}\backslash \{n\}$, $j\in \{1,2,\ldots ,k\}$ and $\epsilon \in \{0,1\}$ such that $\sigma _{\mathbf{v}}(n)=\begin{cases} n & \textrm{if $\mathbf{v}_{j}=\epsilon $}\\
i & \textrm{if $\mathbf{v}_{j}\neq \epsilon $}\end{cases}$, $\sigma _{\mathbf{v}}(i)=\begin{cases}i & \textrm{if $\mathbf{v}_{j}=\epsilon $}\\  n & \textrm{if $\mathbf{v}_{j}\neq \epsilon $}\end{cases}$ and $N(n)=\{i\}$; then $$\sigma _{(\mathbf{v}_{1},\ldots ,\widehat{\mathbf{v}_{j}},\ldots ,\mathbf{v}_{k})}'=\sigma _{\mathbf{v}}|_{\{1,2,\ldots ,n-1\}} \textrm{ for $\mathbf{v}\in \{\mathbf{w}\in \{0,1\}^{k}\,|\,\mathbf{w}_{j}=\epsilon \}$}\;.$$ 

\item[(g)] $k'=k+2$, $n'=n+4$ and there exist two labels $i,j\in \{1,2,\ldots ,n\}$ with $|N(i)|>0$, $|N(j)|>0$ and $N(i)\cap N(j)=\emptyset $, such that 
\begin{xalignat*}{1}
& \left \{\sigma _{\mathbf{v}\times \{(0,0)\}}',\sigma _{\mathbf{v}\times \{(1,1)\}}'\right \}=\left \{\cycle{j,b}\cycle{\sigma _{\mathbf{v}}(i),a}\cycle{\sigma _{\mathbf{v}}(j),i,c,d}\sigma _{\mathbf{v}},\cycle{i,c}\cycle{\sigma _{\mathbf{v}}(j),d}\cycle{\sigma _{\mathbf{v}}(i),j,b,a}\sigma _{\mathbf{v}}\right \}\\
& \left \{\sigma _{\mathbf{v}\times \{(0,1)\}}',\sigma _{\mathbf{v}\times \{(1,0)\}}'\right \}=\left \{\cycle{j,b}\cycle{\sigma _{\mathbf{v}}(j),d}\cycle{\sigma _{\mathbf{v}}(i),a}\cycle{i,c}\sigma _{\mathbf{v}},\cycle{\sigma _{\mathbf{v}}(i),,j,b,a}\cycle{\sigma _{\mathbf{v}}(j),i,c,d}\sigma _{\mathbf{v}}\right \}
\end{xalignat*} for every $\mathbf{v}\in \{0,1\}^{k}$ with $\{a,b,c,d\}=\{n+1,n+2,n+3,n+4\}$. 
\item[(h)] $k'=k+2$, $n'=n+3$ and there exist two labels $i,j\in \{1,2,\ldots ,n\}$ with $|N(i)|>0$, $|N(j)|=0$, such that  
\begin{xalignat*}{1}
& \left \{\sigma _{\mathbf{v}\times \{(0,0)\}}',\sigma _{\mathbf{v}\times \{(1,1)\}}'\right \}=\left \{\cycle{\sigma _{\mathbf{v}}(i),a}\cycle{i,c,b,j}\sigma _{\mathbf{v}},\cycle{\sigma _{\mathbf{v}}(i),j,b,a}\cycle{i,c}\sigma _{\mathbf{v}}\right \}\\
& \left \{\sigma _{\mathbf{v}\times \{(0,1)\}}',\sigma _{\mathbf{v}\times \{(1,0)\}}'\right \}=\left \{\cycle{j,b}\cycle{\sigma _{\mathbf{v}}(i),a}\cycle{i,c}\sigma _{\mathbf{v}},\cycle{i,c,b,a,\sigma _{\mathbf{v}}(i),j}\sigma _{\mathbf{v}}\right \}
\end{xalignat*} for every $\mathbf{v}\in \{0,1\}^{k}$ with $\{a,b,c\}=\{n+1,n+2,n+3\}$.
\item[(i)] $k'=k+2$, $n'=n+2$ and there exist two labels $i,j\in \{1,2,\ldots ,n\}$ with $|N(i)|=|N(j)|=0$, such that
\begin{xalignat*}{1}
& \left \{\sigma _{\mathbf{v}\times \{(0,0)\}}',\sigma _{\mathbf{v}\times \{(1,1)\}}'\right \}=\left \{\cycle{j,i,a,b}\sigma _{\mathbf{v}},\cycle{i,j,b,a}\sigma _{\mathbf{v}}\right \}\\
& \left \{\sigma _{\mathbf{v}\times \{(0,1)\}}',\sigma _{\mathbf{v}\times \{(1,0)\}}'\right \}=\left \{\cycle{a,i}\cycle{j,b}\sigma _{\mathbf{v}},\cycle{a,b}\cycle{i,j}\sigma _{\mathbf{v}}\right \}
\end{xalignat*} for every $\mathbf{v}\in \{0,1\}^{k}$ with $\{a,b\}=\{n+1,n+2\}$.

\item[(j)] $k'=k-2$, $n'=n-4$ and there exist two labels $i,j\in \{1,2,\ldots ,n-4\}$ with $|N(i)|>0$, $|N(j)|>0$ and $N(i)\cap N(j)=\emptyset $, such that 
\begin{xalignat*}{1}
& \left \{\sigma _{\mathbf{v}\times \{(0,0)\}},\sigma _{\mathbf{v}\times \{(1,1)\}}\right \}=\left \{\cycle{j,b}\cycle{\sigma _{\mathbf{v}}'(i),a}\cycle{\sigma _{\mathbf{v}}'(j),i,c,d}\sigma _{\mathbf{v}}',\cycle{i,c}\cycle{\sigma _{\mathbf{v}}'(j),d}\cycle{\sigma _{\mathbf{v}}'(i),j,b,a}\sigma _{\mathbf{v}}'\right \}\\
& \left \{\sigma _{\mathbf{v}\times \{(0,1)\}},\sigma _{\mathbf{v}\times \{(1,0)\}}\right \}=\left \{\cycle{j,b}\cycle{\sigma _{\mathbf{v}}'(j),d}\cycle{\sigma _{\mathbf{v}}'(i),a}\cycle{i,c}\sigma _{\mathbf{v}}',\cycle{\sigma _{\mathbf{v}}'(i),j,b,a}\cycle{\sigma _{\mathbf{v}}'(j),i,c,d}\sigma _{\mathbf{v}}'\right \}
\end{xalignat*} for every $\mathbf{v}\in \{0,1\}^{k-2}$ with $\{a,b,c,d\}=\{n-3,n-2,n-1,n\}$. 
\item[(k)] $k'=k-2$, $n'=n-3$ and there exist two labels $i,j\in \{1,2,\ldots ,n\}$ with $|N(i)|>0$, $|N(j)|=0$, such that 
\begin{xalignat*}{1}
& \left \{\sigma _{\mathbf{v}\times \{(0,0)\}},\sigma _{\mathbf{v}\times \{(1,1)\}}\right \}=\left \{\cycle{\sigma _{\mathbf{v}}'(i),a}\cycle{i,c,b,j}\sigma _{\mathbf{v}}',\cycle{\sigma _{\mathbf{v}}'(i),j,b,a}\cycle{i,c}\sigma _{\mathbf{v}}'\right \}\\
& \left \{\sigma _{\mathbf{v}\times \{(0,1)\}},\sigma _{\mathbf{v}\times \{(1,0)\}}\right \}=\left \{\cycle{j,b}\cycle{\sigma _{\mathbf{v}}'(i),a}\cycle{i,c}\sigma _{\mathbf{v}}',\cycle{i,c,b,a,\sigma _{\mathbf{v}}'(i),j}\sigma _{\mathbf{v}}'\right \}
\end{xalignat*} for every $\mathbf{v}\in \{0,1\}^{k}$ with $\{a,b,c\}=\{n-3,n-2,n-1\}$.
\item[(l)] $k'=k-2$, $n'=n-2$ and there exist two labels $i,j\in \{1,2,\ldots ,n\}$ with $|N(i)|=|N(j)|=0$, such that 
\begin{xalignat*}{1}
& \left \{\sigma _{\mathbf{v}\times \{(0,0)\}},\sigma _{\mathbf{v}\times \{(1,1)\}}\right \}=\left \{\cycle{j,i,a,b}\sigma _{\mathbf{v}}',\cycle{i,j,b,a}\sigma _{\mathbf{v}}'\right \}\\
& \left \{\sigma _{\mathbf{v}\times \{(0,1)\}},\sigma _{\mathbf{v}\times \{(1,0)\}}\right \}=\left \{\cycle{a,i}\cycle{j,b}\sigma _{\mathbf{v}}',\cycle{a,b}\cycle{i,j}\sigma _{\mathbf{v}}'\right \}
\end{xalignat*} for every $\mathbf{v}\in \{0,1\}^{k}$ with $\{a,b\}=\{n-1,n\}$.
\end{itemize}
Two vertex labelings $l\colon \{0,1\}^{k}\to S_{n}$ and $l'\colon \{0,1\}^{k'}\to S_{n'}$ are called \emph{equivalent} if $l'$ can be obtained from $l$ in a finite sequence of transformations. 
\end{definition}

The cartesian product of two hypercube graphs is another hypercube graph: $Q_{k}\Box Q_{l}=Q_{k+l}$. In a type (c) transformation,  a labeling of the $k$-dimensional cube $Q_k$ induces a labeling of $Q_{k+1}=Q_{k}\Box Q_{1}$. Similarly, in a type (e) transformation, a labeling of $Q_k$ induces a labeling of $Q_{k+2}=Q_{k}\Box Q_{2}$. 

\begin{Theorem} The equivalence class of a cube of permutation data is an invariant of the underlying link. 
\end{Theorem}
\begin{proof} First we will prove that any two cubes of permutation data of a given link diagram belong to the same equivalence class, then we will show that the cubes of permutation data of two diagrams, connected by a Reidemeister move, are equivalent.  

Let $V=\{1,2,\ldots ,n\}$ be the set of labels for a link diagram $\mathcal{D}$ with $k$ crossings. Choose any triangulations $\tau _{1},\tau _{2}\colon E(\mathcal{D})\to V$ of $\mathcal{D}$ and let $o_{i}$ be a web orientation of the cube of smoothings for $(\mathcal{D},\tau _{i})$ for $i=1,2$. Denote by $\mathcal{C}_{1}(\mathcal{D})$ the cube of permutation data, associated with $(\mathcal{D},\tau _1)$ and $o_1$, and let $\{\sigma _{\mathbf{v}}\in S_{n}\,|\,\mathbf{v}\in \{0,1\}^{k}\}$ be the corresponding vertex labeling of $Q_k$. Similarly, denote by $\mathcal{C}_{2}(\mathcal{D})$ the cube of permutation data, associated with $(\mathcal{D},\tau _2)$ and $o_2$, and let $\{\sigma _{\mathbf{v}}'\in S_{n}\,|\,\mathbf{v}\in \{0,1\}^{k}\}$ be the corresponding vertex labeling of $Q_k$. Then $\rho =\tau _{2}\circ \tau _{1}^{-1}\colon V\to V$ is a bijection on the set of labels. If $o_{1}=o_2$, then $\sigma _{\mathbf{v}}'=\rho \circ \sigma _{\mathbf{v}}\circ \rho ^{-1}$ and thus $\mathcal{C}_{2}(\mathcal{D})$ is a transformation of $\mathcal{C}_{1}(\mathcal{D})$. If $o_{1}\neq o_2$, then $o_1$ and $o_2$ are opposite web orientations. For every $\mathbf{v}\in \{0,1\}^{k}$, there exists an order 2 element $\zeta _{\mathbf{v}}\in G_{\mathbf{v}}$, for which $\zeta _{\mathbf{v}}\circ \sigma _{\mathbf{v}}\circ \zeta _{\mathbf{v}}^{-1}=\sigma _{\mathbf{v}}^{-1}$. Thus, $\mathcal{C}_{2}(\mathcal{D})$ is obtained from $\mathcal{C}_{1}(\mathcal{D})$ by a transformation of type (a), followed by a transformation of type (b).  

It remains to show that the equivalence class of a cube of permutations is invariant under the Reidemeister moves. Let $(\mathcal{D}_{1},\tau _1)$ and $(\mathcal{D}_{2},\tau _2)$ be two triangulated link diagrams that are identical outside a disk $D$, while inside $D$ they differ by a Reidemeister move. Denote by $\{1,2,\ldots ,n\}$ (resp. $\{1,2,\ldots ,n'\}$) the set of labels of $\mathcal{D}_{1}$ (resp. $\mathcal{D}_2$). 
\begin{itemize}
\item Suppose that $(\mathcal{D}_{1},\tau _1)$ and $(\mathcal{D}_{2},\tau _2)$ differ by the first Reidemeister move $R_1$ (or $R_{1}'$), during which the edge with label $i$ of the diagram $\mathcal{D}_1$ forms a new crossing. It follows from Lemma \ref{lemma2} that if $N(i)>0$, $\mathcal{C}(\mathcal{D}_{2})$ is a type (c) transformation of $\mathcal{C}(\mathcal{D}_{1})$, while if $N(i)=0$, $\mathcal{C}(\mathcal{D}_{2})$ is a type (d) transformation of $\mathcal{C}(\mathcal{D}_{1})$. 
\item Suppose that the diagram $(\mathcal{D}_{2},\tau _2)$ is obtained from $(\mathcal{D}_{1},\tau _1)$ by the inverse of the Reidemeister move $R_1$ (or $R_{1}'$), during which a crossing   $c$, incident with only three (or possibly two) edges of the diagram $\mathcal{D}_1$, disappears. If $c$ is incident with three edges, a suitable type (a) transformation of $\mathcal{C}(\mathcal{D}_{1})$ will relabel those edges into $i$, $n-1$ and $n$, so that the edge labeled by $n$ has both endpoints at the crossing $c$. If $c$ is incident with two edges, a type (a) transformation of $\mathcal{C}(\mathcal{D}_{1})$ will relabel those edges into $i$ and $n$. Then Lemma \ref{lemma3} implies that $\mathcal{C}(\mathcal{D}_{2})$ is a type (e) transformation of $\mathcal{C}(\mathcal{D}_{1})$ if $c$ is incident with three edges, or a type (f) transformation of $\mathcal{C}(\mathcal{D}_{1})$ if $c$ is incident with two edges.  

\item Suppose that $(\mathcal{D}_{2},\tau _2)$ is obtained from $(\mathcal{D}_{1},\tau _1)$ by the second Reidemeister move $R_2$ (or $R_{2}'$), during which the edge with label $i$ passes over (or under) an edge with label $j$. Using a type (a) transformation of $\mathcal{C}(\mathcal{D}_{1})$, we may assume that the edges $E(\mathcal{D}_{2})\backslash E(\mathcal{D}_{1})$ are labeled by elements of the set $\{1,2,\ldots ,n'\}\backslash \{1,2,\ldots ,n\}$ in a suitable way. Lemma \ref{lemma4} (part (a)) implies that if $|N(i)|>0$, $|N(j)|>0$ and $N(i)\cap N(j)=\emptyset $, $\mathcal{C}(\mathcal{D}_{2})$ is a type (g) transformation of $\mathcal{C}(\mathcal{D}_{1})$. Part (b) of Lemma \ref{lemma4} implies that if $|N(i)|>0$ and $|N(j)|=0$ (resp. $|N(i)|=0$ and $|N(j)|>0$), $\mathcal{C}(\mathcal{D}_{2})$ is a type (h) transformation of $\mathcal{C}(\mathcal{D}_{1})$ (in the second case, labels $i$ and $j$ need to be switched by a type (a) transformation before applying a type (h) transformation). If $|N(i)|=|N(j)|=0$, part (d) of Lemma \ref{lemma4} implies that $\mathcal{C}(\mathcal{D}_{2})$ is a type (i) transformation of $\mathcal{C}(\mathcal{D}_{1})$. If $|N(i)\cap N(j)|>0$, then both new crossings created by the Reidemeister move $R_2$ (or $R_{2}'$) may also be created by an $R_{1}$ move and an $R_{1}'$ move, therefore $\mathcal{C}(\mathcal{D}_2)$ is obtained from $\mathcal{C}(\mathcal{D}_1)$ by two transformations of type (c).     

\item Suppose that $(\mathcal{D}_{2},\tau _2)$ is obtained from $(\mathcal{D}_{1},\tau _1)$ by the inverse of the second Reidemeister move $R_2$ (or $R_{2}'$), during which two crossings, connected by two edges with labels $i$ and $j$, disappear. Applying a type (a) transformation, we may assume that the edges $E(\mathcal{D}_{1})\backslash E(\mathcal{D}_{2})$ are labeled by elements of the set $\{1,2,\ldots ,n\}\backslash \{1,2,\ldots ,n'\}$ in a suitable way. We have several possibilities, depending on the adjacent edges of $i$ and $j$ in the triangulated diagram $(\mathcal{D}_{2},\tau _2)$. Lemma \ref{lemma4} implies that if $|N(i)|>0$, $|N(j)|>0$ and $N(i)\cap N(j)=\emptyset $ in the diagram $\mathcal{D}_2$, $\mathcal{C}(\mathcal{D}_{2})$ is a type (j) transformation of $\mathcal{C}(\mathcal{D}_{1})$. If $|N(i)|>0$ and $|N(j)|=0$ (resp. $|N(i)|=0$ and $|N(j)|>0$), $\mathcal{C}(\mathcal{D}_{2})$ is a type (k) transformation of $\mathcal{C}(\mathcal{D}_{1})$ (in the second case, labels $i$ and $j$ need to be switched by a type (a) transformation before applying a type (k) transformation). If $|N(i)|=|N(j)|=0$, $\mathcal{C}(\mathcal{D}_{2})$ is a type (l) transformation of $\mathcal{C}(\mathcal{D}_{1})$. If $|N(i)\cap N(j)|>0$, then both crossings deleted by the Reidemeister move $R_{2}^{-1}$ (or $R_{2}'^{-1}$) may be also be removed by an $R_{1}^{-1}$ move and an $R_{1}'^{-1}$ move, therefore $\mathcal{C}(\mathcal{D}_{2})$ is obtained from $\mathcal{C}(\mathcal{D}_{1})$ by a sequence of two transformations of type (e).     

\item Suppose that $(\mathcal{D}_{1},\tau _1)$ and $(\mathcal{D}_{2},\tau _2)$ differ by the third Reidemeister move $R_3$, shown in Figure \ref{fig11}. The disk $D$ contains three crossings, connected by a triangle of edges with labels $g$, $h$ and $i$. Denote by $\mathcal{C}(\mathcal{D}_1)$ and $\mathcal{C}(\mathcal{D}_2)$ the cubes of permutation data, belonging to $(\mathcal{D}_{1},\tau _1)$ and $(\mathcal{D}_{2},\tau _2)$, respectively. Let $\{\sigma _{\mathbf{v}}\,|\,\mathbf{v}\in \{0,1\}^{k}\}$ (resp. $\{\sigma _{\mathbf{v}}'\,|\,\mathbf{v}\in \{0,1\}^{k}\}$) be the vertex labeling of $Q_k$, corresponding to $\mathcal{C}(\mathcal{D}_1)$ (resp. $\mathcal{C}(\mathcal{D}_2)$). 

If $|N(g)|=|N(h)|=|N(i)|=2$ and $N(g)$, $N(h)$ and $N(i)$ are pairwise disjoint in the diagram $\mathcal{D}_1$, there are nine vertices inside $D$, and the local smoothings of $\mathcal{D}_1$ and $\mathcal{D}_2$ are shown in Figure \ref{fig9}. A local comparison of both cubes of permutation data gives the following relationships: 
\begin{xalignat*}{2}
& \sigma _{000}'=\cycle{a,d}\cycle{b,e}\cycle{c,f}\sigma _{000}\cycle{a,d}\cycle{b,e}\cycle{c,f} & \sigma _{100}'=\cycle{a,d}\cycle{b,e}\cycle{c,f}\sigma _{100}\cycle{a,d}\cycle{b,e}\cycle{c,f}\\
& \sigma _{011}'=\cycle{a,d}\cycle{b,e}\cycle{c,f}\sigma _{011}\cycle{a,d}\cycle{b,e}\cycle{c,f} & \sigma _{111}'=\cycle{a,d}\cycle{b,e}\cycle{c,f}\sigma _{111}\cycle{a,d}\cycle{b,e}\cycle{c,f}\\
& \sigma _{001}'=\cycle{a,d}\cycle{b,e}\cycle{c,f}\sigma _{010}\cycle{a,d}\cycle{b,e}\cycle{c,f} & \sigma _{010}'=\cycle{a,d}\cycle{b,e}\cycle{c,f}\sigma _{001}\cycle{a,d}\cycle{b,e}\cycle{c,f}\\
& \sigma _{101}'=\cycle{a,d}\cycle{b,e}\cycle{c,f}\sigma _{110}\cycle{a,d}\cycle{b,e}\cycle{c,f} & \sigma _{110}'=\cycle{a,d}\cycle{b,e}\cycle{c,f}\sigma _{101}\cycle{a,d}\cycle{b,e}\cycle{c,f}
\end{xalignat*}
To simplify notation, we supressed the indices of all other crossings except the triple of crossings inside $D$. Denote by $\eta \colon \{1,2,\ldots ,k\}\to \{1,2\ldots ,k\}$ the transposition of indices, corresponding to the second and the third crossing in this triple. Observe that the words $000$, $100$, $011$ and $111$ are all invariant to the exchange of the second and third letter. We may thus conclude that 
\begin{equation}\label{eq}
\sigma _{\mathbf{v}}'=\cycle{a,d}\cycle{b,e}\cycle{c,f}\sigma _{\widehat{\eta }(\mathbf{v})}\cycle{a,d}\cycle{b,e}\cycle{c,f}\end{equation} for every $\mathbf{v}\in \{0,1\}^{k}$. Therefore, the cube of permutation data $\mathcal{C}(\mathcal{D}_2)$ is a type (a) transformation of $\mathcal{C}(\mathcal{D}_1)$. 

If either of $|N(g)|$, $|N(h)|$ or $|N(i)|$ is equal to 1, the labels of the adjacent edges of $g$, $h$ or $i$ become identified and the corresponding transposition is omitted from the expression \ref{eq}. Thus, for example, if $|N(g)|=1$, we obtain $b=e$ and $\sigma _{\mathbf{v}}'=\cycle{a,d}\cycle{c,f}\sigma _{\widehat{\eta }(\mathbf{v})}\cycle{a,d}\cycle{c,f}$.

If the sets $N(g)$, $N(h)$ and $N(i)$ in the diagram $\mathcal{D}_1$ are not pairwise disjoint, certain pairs of labels $a$, $b$, $c$, $d$, $e$ and $f$ become identified, and the expression \ref{eq} changes accordingly. If, for example, $b=c$, we obtain $\sigma _{\mathbf{v}}'=\cycle{a,d}\cycle{b,f,e}\sigma _{\widehat{\eta }(\mathbf{v})}\cycle{b,e,f}\cycle{a,d}$. We may thus conclude that in all possible cases, the cube of permutation data $\mathcal{C}(\mathcal{D}_2)$ is a type (a) transformation of $\mathcal{C}(\mathcal{D}_1)$.
% In some cases, all three crossings may be removed using a sequence of Reidemeister moves $R_{1}^{-1}$, $R_{1}'^{-1}$, $R_{2}^{-1}$ or $R_{2}'^{-1}$ (this happens when $a=b$, $d=e$, $c=d$ or $a=f$).
\end{itemize}
\end{proof}

\begin{example} A cube of permutation data for an $n$-component unlink is equivalent to a point (a $0$-dimensional hypercube), labeled by the identity permutation. 
\end{example}

We conclude by some open questions and ideas for future research.

\begin{enumerate}
\item Does there exist a simple characterization of a cube of permutation data of an unlink diagram?
\item The set of vertices in a link diagram triangulation corresponds bijectively with the generating set of the fundamental biquandle of the same link. Is there a useful interpretation of our invariant in terms of biquandles?  
\item An exploration of our combinatorial invariant, based on a suitable computer implementation, could lead to a better understanding of its properties, relationship with other invariants and possible applications in knot theory.
\end{enumerate}

\section*{Acknowledgements}

The author was supported by the Slovenian Research Agency grant P1-0292. Moreover, the author would like to thank her colleague dr. Primo\v z \v Sparl for several useful comments on an earlier version of this paper.

 \begin{figure}[H]
\begin{tikzpicture}
\matrix [column sep=1cm]
{
\node {}; & \node (001) {\begin{tikzpicture}
\draw[black, dashed] (0,0) circle (40pt);
\draw[black, thick][-] (-0.7,1.2) .. controls (-0.75,1) and (-0.35,0.4) .. (0,0.35)
                                                    .. controls (0.5,-0.2) and (0.35,-0.9) .. (0.7,-1.24);
\draw[black, thick][-] (-1.4,0.1) .. controls (-0.5,0.3) and (-0.2,-0.3) .. (-0.3,-0.15)
                                                    .. controls (-0.1,-0.5) and (-0.55,-1) .. (-0.7,-1.24);
\draw[black, thick][-] (1.4,0.1) .. controls (0.5,0.1) and (0.5,0.6) .. (0.7,1.2);
\filldraw[black] (-0.5,-0.93) circle (2pt) node[anchor=east]{$a$};
\filldraw[black] (0.5,-0.93) circle (2pt) node[anchor=west]{$b$};
\filldraw[black] (1.05,0.15) circle (2pt) node[anchor=north]{$c$};
\filldraw[black] (0.6,0.85) circle (2pt) node[anchor=east]{$d$};
\filldraw[black] (-1.05,0.15) circle (2pt) node[anchor=north]{$f$};
\filldraw[black] (-0.6,0.85) circle (2pt) node[anchor=west]{$e$};
\filldraw[black] (-0.3,-0.15) circle (2pt) node[anchor=east]{$g$};
\filldraw[black] (0.3,-0.15) circle (2pt) node[anchor=west]{$h$};
\filldraw[black] (0,0.35) circle (2pt) node[anchor=south]{$i$};
\end{tikzpicture}}; & 
\node (011) {\begin{tikzpicture}
\draw[black, dashed] (0,0) circle (40pt);
\draw[black, thick][-] (-0.7,-1.24) .. controls (-0.4,-1) and (-0.35,0.2) .. (0,0.35)
                                                       .. controls (0.35,0.2) and (0.4,-1) .. (0.7,-1.24);
\draw[black, thick][-] (-1.4,0.1) .. controls (-0.5,0.1) and (-0.5,0.6) .. (-0.7,1.2);
\draw[black, thick][-] (1.4,0.1) .. controls (0.5,0.1) and (0.5,0.6) .. (0.7,1.2);
\filldraw[black] (-0.5,-0.93) circle (2pt) node[anchor=east]{$a$};
\filldraw[black] (0.5,-0.93) circle (2pt) node[anchor=west]{$b$};
\filldraw[black] (1.05,0.15) circle (2pt) node[anchor=north]{$c$};
\filldraw[black] (0.6,0.85) circle (2pt) node[anchor=east]{$d$};
\filldraw[black] (-1.05,0.15) circle (2pt) node[anchor=north]{$f$};
\filldraw[black] (-0.6,0.85) circle (2pt) node[anchor=west]{$e$};
\filldraw[black] (-0.3,-0.15) circle (2pt) node[anchor=east]{$g$};
\filldraw[black] (0.3,-0.15) circle (2pt) node[anchor=west]{$h$};
\filldraw[black] (0,0.35) circle (2pt) node[anchor=south]{$i$};
\end{tikzpicture}}; & \node {}; \\
\node (000) {\begin{tikzpicture}
\draw[black, dashed] (0,0) circle (40pt);
\draw[black, thick][-] (-0.7,-1.24) .. controls (-0.3,-0.5) and (0.3,-0.5) .. (0.7,-1.24);
\draw[black, thick][-] (-1.4,0.1) .. controls (-0.5,0.3) and (-0.3,-0.45) .. (0.3,-0.15)
                                                 .. controls (0.3,0.6) and (-0.45,0.1) .. (-0.7,1.2);
\draw[black, thick][-] (1.4,0.1) .. controls (0.5,0.1) and (0.5,0.6) .. (0.7,1.2);
\filldraw[black] (-0.5,-0.93) circle (2pt) node[anchor=east]{$a$};
\filldraw[black] (0.5,-0.93) circle (2pt) node[anchor=west]{$b$};
\filldraw[black] (1.05,0.15) circle (2pt) node[anchor=north]{$c$};
\filldraw[black] (0.6,0.85) circle (2pt) node[anchor=east]{$d$};
\filldraw[black] (-1.05,0.15) circle (2pt) node[anchor=north]{$f$};
\filldraw[black] (-0.6,0.85) circle (2pt) node[anchor=west]{$e$};
\filldraw[black] (-0.3,-0.15) circle (2pt) node[anchor=north]{$g$};
\filldraw[black] (0.3,-0.15) circle (2pt) node[anchor=north]{$h$};
\filldraw[black] (0,0.35) circle (2pt) node[anchor=south]{$i$};
\end{tikzpicture}}; &
\node (010) {\begin{tikzpicture}
\draw[black, dashed] (0,0) circle (40pt);
\draw[black, thick] (0,0) circle (10pt);
\draw[black, thick][-] (-0.7,-1.24) .. controls (-0.3,-0.5) and (0.3,-0.5) .. (0.7,-1.24);
\draw[black, thick][-] (-1.4,0.1) .. controls (-0.5,0.1) and (-0.5,0.6) .. (-0.7,1.2);
\draw[black, thick][-] (1.4,0.1) .. controls (0.5,0.1) and (0.5,0.6) .. (0.7,1.2);
\filldraw[black] (-0.5,-0.93) circle (2pt) node[anchor=east]{$a$};
\filldraw[black] (0.5,-0.93) circle (2pt) node[anchor=west]{$b$};
\filldraw[black] (1.05,0.15) circle (2pt) node[anchor=north]{$c$};
\filldraw[black] (0.6,0.85) circle (2pt) node[anchor=east]{$d$};
\filldraw[black] (-1.05,0.15) circle (2pt) node[anchor=north]{$f$};
\filldraw[black] (-0.6,0.85) circle (2pt) node[anchor=west]{$e$};
\filldraw[black] (-0.3,-0.15) circle (2pt) node[anchor=east]{$g$};
\filldraw[black] (0.3,-0.15) circle (2pt) node[anchor=west]{$h$};
\filldraw[black] (0,0.35) circle (2pt) node[anchor=south]{$i$};
\end{tikzpicture}}; &
\node (101) {\begin{tikzpicture}
\draw[black, dashed] (0,0) circle (40pt);
\draw[black, thick][-] (-0.7,1.2) .. controls (-0.55,0.4) and (-0.2,0.35) .. (0,0.35)
                                                    .. controls (0.2,0.35) and (0.55,0.4) .. ( (0.7,1.2);
\draw[black, thick][-] (-1.4,0.1) .. controls (-0.5,0.3) and (-0.2,-0.3) .. (-0.3,-0.15)
                                                    .. controls (-0.1,-0.5) and (-0.55,-1) .. (-0.7,-1.24);
\draw[black, thick][-] (1.4,0.1) .. controls (0.5,0.3) and (0.2,-0.3) .. (0.3,-0.15)
                                                    .. controls (0.1,-0.5) and (0.55,-1) .. (0.7,-1.24);
\filldraw[black] (-0.5,-0.93) circle (2pt) node[anchor=east]{$a$};
\filldraw[black] (0.5,-0.93) circle (2pt) node[anchor=west]{$b$};
\filldraw[black] (1.05,0.15) circle (2pt) node[anchor=north]{$c$};
\filldraw[black] (0.6,0.85) circle (2pt) node[anchor=east]{$d$};
\filldraw[black] (-1.05,0.15) circle (2pt) node[anchor=north]{$f$};
\filldraw[black] (-0.6,0.85) circle (2pt) node[anchor=west]{$e$};
\filldraw[black] (-0.3,-0.15) circle (2pt) node[anchor=east]{$g$};
\filldraw[black] (0.3,-0.15) circle (2pt) node[anchor=west]{$h$};
\filldraw[black] (0,0.35) circle (2pt) node[anchor=south]{$i$};
\end{tikzpicture}};  &
\node (111) {\begin{tikzpicture}
\draw[black, dashed] (0,0) circle (40pt);
\draw[black, thick][-] (0.7,1.2) .. controls (0.75,1) and (0.35,0.4) .. (0,0.35)
                                                    .. controls (-0.5,-0.2) and (-0.35,-0.9) .. (-0.7,-1.24);
\draw[black, thick][-] (-1.4,0.1) .. controls (-0.5,0.1) and (-0.5,0.6) .. (-0.7,1.2);
\draw[black, thick][-] (1.4,0.1) .. controls (0.5,0.3) and (0.2,-0.3) .. (0.3,-0.15)
                                                    .. controls (0.1,-0.5) and (0.55,-1) .. (0.7,-1.24);
\filldraw[black] (-0.5,-0.93) circle (2pt) node[anchor=east]{$a$};
\filldraw[black] (0.5,-0.93) circle (2pt) node[anchor=west]{$b$};
\filldraw[black] (1.05,0.15) circle (2pt) node[anchor=north]{$c$};
\filldraw[black] (0.6,0.85) circle (2pt) node[anchor=east]{$d$};
\filldraw[black] (-1.05,0.15) circle (2pt) node[anchor=north]{$f$};
\filldraw[black] (-0.6,0.85) circle (2pt) node[anchor=west]{$e$};
\filldraw[black] (-0.3,-0.15) circle (2pt) node[anchor=east]{$g$};
\filldraw[black] (0.3,-0.15) circle (2pt) node[anchor=west]{$h$};
\filldraw[black] (0,0.35) circle (2pt) node[anchor=south]{$i$};
\end{tikzpicture}}; \\
\node {}; & \node (100) {\begin{tikzpicture}
\draw[black, dashed] (0,0) circle (40pt);
\draw[black, thick][-] (-0.7,1.2) .. controls (-0.55,0.4) and (-0.2,0.35) .. (0,0.35)
                                                    .. controls (0.2,0.35) and (0.55,0.4) .. ( (0.7,1.2);
\draw[black, thick][-] (-0.7,-1.24) .. controls (-0.3,-0.5) and (0.3,-0.5) .. (0.7,-1.24);
\draw[black, thick][-] (-1.4,0.1) .. controls (-0.5,0.3) and (-0.3,-0.3) .. (0,-0.3)
                                                    .. controls (0.3,-0.3) and (0.5,0.3) .. (1.4,0.1);
\filldraw[black] (-0.5,-0.93) circle (2pt) node[anchor=east]{$a$};
\filldraw[black] (0.5,-0.93) circle (2pt) node[anchor=west]{$b$};
\filldraw[black] (1.05,0.15) circle (2pt) node[anchor=north]{$c$};
\filldraw[black] (0.6,0.85) circle (2pt) node[anchor=east]{$d$};
\filldraw[black] (-1.05,0.15) circle (2pt) node[anchor=north]{$f$};
\filldraw[black] (-0.6,0.85) circle (2pt) node[anchor=west]{$e$};
\filldraw[black] (-0.3,-0.15) circle (2pt) node[anchor=north]{$g$};
\filldraw[black] (0.3,-0.15) circle (2pt) node[anchor=north]{$h$};
\filldraw[black] (0,0.35) circle (2pt) node[anchor=south]{$i$};
\end{tikzpicture}};  & 
\node (110) {\begin{tikzpicture}
\draw[black, dashed] (0,0) circle (40pt);
\draw[black, thick][-] (-0.7,-1.24) .. controls (-0.3,-0.5) and (0.3,-0.5) .. (0.7,-1.24);
\draw[black, thick][-] (-1.4,0.1) .. controls (-0.5,0.1) and (-0.5,0.6) .. (-0.7,1.2);
\draw[black, thick][-] (1.4,0.1) .. controls (0.5,0.3) and (0.3,-0.45) .. (-0.3,-0.15)
                                                 .. controls (-0.3,0.6) and (0.45,0.1) .. (0.7,1.2);
\filldraw[black] (-0.5,-0.93) circle (2pt) node[anchor=east]{$a$};
\filldraw[black] (0.5,-0.93) circle (2pt) node[anchor=west]{$b$};
\filldraw[black] (1.05,0.15) circle (2pt) node[anchor=north]{$c$};
\filldraw[black] (0.6,0.85) circle (2pt) node[anchor=east]{$d$};
\filldraw[black] (-1.05,0.15) circle (2pt) node[anchor=north]{$f$};
\filldraw[black] (-0.6,0.85) circle (2pt) node[anchor=west]{$e$};
\filldraw[black] (-0.3,-0.15) circle (2pt) node[anchor=east]{$g$};
\filldraw[black] (0.3,-0.15) circle (2pt) node[anchor=west]{$h$};
\filldraw[black] (0,0.35) circle (2pt) node[anchor=south]{$i$};
\end{tikzpicture}}; \\
};
\draw[blue,thick][->] (000.east) -- (001.west); 
\draw[blue,thick][->] (000.east) -- (010.west); 
\draw[blue,thick][->] (000.east) -- (100.west); 
\draw[blue,thick][->] (011.east) -- (111.west); 
\draw[blue,thick][->] (101.east) -- (111.west); 
\draw[blue,thick][->] (110.east) -- (111.west); 
\draw[blue,thick][->] (001.east) -- (011.west); 
\draw[blue,thick][->] (001.east) -- (101.west); 
\draw[blue,thick][->] (010.east) -- (011.west); 
\draw[blue,thick][->] (010.east) -- (110.west); 
\draw[blue,thick][->] (100.east) -- (101.west); 
\draw[blue,thick][->] (100.east) -- (110.west); 
\end{tikzpicture}
\begin{tikzpicture}
\matrix [column sep=1cm]
{
\node {}; & \node (001) {\begin{tikzpicture}
\draw[black, dashed] (0,0) circle (40pt);
\draw[black, thick] (0,0) circle (10pt);
\draw[black, thick][-] (-0.7,1.2) .. controls (-0.6,0.3) and (0.6,0.3) .. (0.7,1.2);                                            
\draw[black, thick][-] (-1.4,0.1) .. controls (-0.5,0.3) and (-0.1,-0.5) .. (-0.7,-1.24);
\draw[black, thick][-] (1.4,0.1) .. controls (0.5,0.3) and (0.1,-0.5) .. (0.7,-1.24);
\filldraw[black] (-0.5,-0.93) circle (2pt) node[anchor=east]{$a$};
\filldraw[black] (0.5,-0.93) circle (2pt) node[anchor=west]{$b$};
\filldraw[black] (1.05,0.15) circle (2pt) node[anchor=north]{$c$};
\filldraw[black] (0.6,0.85) circle (2pt) node[anchor=east]{$d$};
\filldraw[black] (-1.05,0.15) circle (2pt) node[anchor=north]{$f$};
\filldraw[black] (-0.6,0.85) circle (2pt) node[anchor=west]{$e$};
\filldraw[black] (0.3,0.15) circle (2pt) node[anchor=west]{$g$};
\filldraw[black] (-0.3,0.15) circle (2pt) node[anchor=east]{$h$};
\filldraw[black] (0,-0.35) circle (2pt) node[anchor=north]{$i$};
\end{tikzpicture}}; & 
\node (011) {\begin{tikzpicture}
\draw[black, dashed] (0,0) circle (40pt);
\draw[black, thick][-] (-0.7,1.2) .. controls (-0.65,1) and (-0.26,-0.2) .. (0,-0.35)
                                                       .. controls (0.26,-0.2) and (0.65,1) .. (0.7,1.2);
\draw[black, thick][-] (-1.4,0.1) .. controls (-0.5,0.3) and (-0.1,-0.5) .. (-0.7,-1.24);
\draw[black, thick][-] (1.4,0.1) .. controls (0.5,0.3) and (0.1,-0.5) .. (0.7,-1.24);
\filldraw[black] (-0.5,-0.93) circle (2pt) node[anchor=east]{$a$};
\filldraw[black] (0.5,-0.93) circle (2pt) node[anchor=west]{$b$};
\filldraw[black] (1.05,0.15) circle (2pt) node[anchor=north]{$c$};
\filldraw[black] (0.6,0.85) circle (2pt) node[anchor=east]{$d$};
\filldraw[black] (-1.05,0.15) circle (2pt) node[anchor=north]{$f$};
\filldraw[black] (-0.6,0.85) circle (2pt) node[anchor=west]{$e$};
\filldraw[black] (0.3,0.15) circle (2pt) node[anchor=west]{$g$};
\filldraw[black] (-0.3,0.15) circle (2pt) node[anchor=east]{$h$};
\filldraw[black] (0,-0.35) circle (2pt) node[anchor=north]{$i$};
\end{tikzpicture}}; & \node {}; \\
\node (000) {\begin{tikzpicture}
\draw[black, dashed] (0,0) circle (40pt);
\draw[black, thick][-] (-0.7,1.2) .. controls (-0.6,0.3) and (0.6,0.3) .. (0.7,1.2);                                            
\draw[black, thick][-] (-1.4,0.1) .. controls (-0.5,0.3) and (-0.1,-0.5) .. (-0.7,-1.24);
\draw[black, thick][-] (1.4,0.1) .. controls (0.3,0.3) and (-0.2,0.05) .. (-0.3,0.15)
                                                .. controls (-0.1,-0.4) and (0.5,-0.9) .. (0.7,-1.24);
\filldraw[black] (-0.5,-0.93) circle (2pt) node[anchor=east]{$a$};
\filldraw[black] (0.5,-0.93) circle (2pt) node[anchor=west]{$b$};
\filldraw[black] (1.05,0.15) circle (2pt) node[anchor=north]{$c$};
\filldraw[black] (0.6,0.85) circle (2pt) node[anchor=east]{$d$};
\filldraw[black] (-1.05,0.15) circle (2pt) node[anchor=north]{$f$};
\filldraw[black] (-0.6,0.85) circle (2pt) node[anchor=west]{$e$};
\filldraw[black] (0.3,0.15) circle (2pt) node[anchor=north]{$g$};
\filldraw[black] (-0.3,0.15) circle (2pt) node[anchor=east]{$h$};
\filldraw[black] (0,-0.35) circle (2pt) node[anchor=north]{$i$};
\end{tikzpicture}}; &
\node (010) {\begin{tikzpicture}
\draw[black, dashed] (0,0) circle (40pt);
\draw[black, thick][-] (-0.7,1.2) .. controls (-0.75,1) and (-0.4,0.4) .. (0,-0.35)
                                                    .. controls (0.3,-0.7) .. (0.7,-1.24);
\draw[black, thick][-] (1.4,0.1) .. controls (0.9,0.2) .. (0.3,0.15)
                                                .. controls (0.5,0.6) .. (0.7,1.2);
\draw[black, thick][-] (-1.4,0.1) .. controls (-0.5,0.3) and (-0.1,-0.5) .. (-0.7,-1.24);
\filldraw[black] (-0.5,-0.93) circle (2pt) node[anchor=east]{$a$};
\filldraw[black] (0.5,-0.93) circle (2pt) node[anchor=west]{$b$};
\filldraw[black] (1.05,0.15) circle (2pt) node[anchor=north]{$c$};
\filldraw[black] (0.6,0.85) circle (2pt) node[anchor=east]{$d$};
\filldraw[black] (-1.05,0.15) circle (2pt) node[anchor=north]{$f$};
\filldraw[black] (-0.6,0.85) circle (2pt) node[anchor=west]{$e$};
\filldraw[black] (0.3,0.15) circle (2pt) node[anchor=north]{$g$};
\filldraw[black] (-0.3,0.15) circle (2pt) node[anchor=east]{$h$};
\filldraw[black] (0,-0.35) circle (2pt) node[anchor=north]{$i$};
\end{tikzpicture}}; &
\node (101) {\begin{tikzpicture}
\draw[black, dashed] (0,0) circle (40pt);
\draw[black, thick][-] (0.7,1.2) .. controls (0.6,0.3) and (-0.6,0.3) .. (-0.7,1.2);                                            
\draw[black, thick][-] (1.4,0.1) .. controls (0.5,0.3) and (0.1,-0.5) .. (0.7,-1.24);
\draw[black, thick][-] (-1.4,0.1) .. controls (-0.3,0.3) and (0.2,0.05) .. (0.3,0.15)
                                                .. controls (0.1,-0.4) and (-0.5,-0.9) .. (-0.7,-1.24);
\filldraw[black] (-0.5,-0.93) circle (2pt) node[anchor=west]{$a$};
\filldraw[black] (0.5,-0.93) circle (2pt) node[anchor=west]{$b$};
\filldraw[black] (1.05,0.15) circle (2pt) node[anchor=north]{$c$};
\filldraw[black] (0.6,0.85) circle (2pt) node[anchor=east]{$d$};
\filldraw[black] (-1.05,0.15) circle (2pt) node[anchor=north]{$f$};
\filldraw[black] (-0.6,0.85) circle (2pt) node[anchor=west]{$e$};
\filldraw[black] (0.3,0.15) circle (2pt) node[anchor=west]{$g$};
\filldraw[black] (-0.3,0.15) circle (2pt) node[anchor=north]{$h$};
\filldraw[black] (0,-0.35) circle (2pt) node[anchor=north]{$i$};
\end{tikzpicture}};  &
\node (111) {\begin{tikzpicture}
\draw[black, dashed] (0,0) circle (40pt);
\draw[black, thick][-] (0.7,1.2) .. controls (0.75,1) and (0.4,0.4) .. (0,-0.35)
                                                    .. controls (-0.3,-0.7) .. (-0.7,-1.24);
\draw[black, thick][-] (-1.4,0.1) .. controls (-0.9,0.2) .. (-0.3,0.15)
                                                .. controls (-0.5,0.6) .. (-0.7,1.2);
\draw[black, thick][-] (1.4,0.1) .. controls (0.5,0.3) and (0.1,-0.5) .. (0.7,-1.24);
\filldraw[black] (-0.5,-0.93) circle (2pt) node[anchor=east]{$a$};
\filldraw[black] (0.5,-0.93) circle (2pt) node[anchor=west]{$b$};
\filldraw[black] (1.05,0.15) circle (2pt) node[anchor=north]{$c$};
\filldraw[black] (0.6,0.85) circle (2pt) node[anchor=east]{$d$};
\filldraw[black] (-1.05,0.15) circle (2pt) node[anchor=north]{$f$};
\filldraw[black] (-0.6,0.85) circle (2pt) node[anchor=west]{$e$};
\filldraw[black] (0.3,0.15) circle (2pt) node[anchor=west]{$g$};
\filldraw[black] (-0.3,0.15) circle (2pt) node[anchor=north]{$h$};
\filldraw[black] (0,-0.35) circle (2pt) node[anchor=north]{$i$};
\end{tikzpicture}}; \\
\node {}; & \node (100) {\begin{tikzpicture}
\draw[black, dashed] (0,0) circle (40pt);
\draw[black, thick][-] (-0.7,-1.24) .. controls (-0.3,-0.55) and (-0.2,-0.4) .. (0,-0.35)
                                                    .. controls (0.2,-0.4) and (0.3,-0.55) .. ( (0.7,-1.24);
\draw[black, thick][-] (-0.7,1.2) .. controls (-0.63,0.4) and (0.63,0.4) .. (0.7,1.2);
\draw[black, thick][-] (-1.4,0.1) .. controls (-0.8,0.15) and (-0.3,0.15) .. (0,0.15)
                                                    .. controls (0.3,0.15) and (0.8,0.15) .. (1.4,0.1);
\filldraw[black] (-0.5,-0.93) circle (2pt) node[anchor=east]{$a$};
\filldraw[black] (0.5,-0.93) circle (2pt) node[anchor=west]{$b$};
\filldraw[black] (1.05,0.15) circle (2pt) node[anchor=north]{$c$};
\filldraw[black] (0.6,0.85) circle (2pt) node[anchor=north]{$d$};
\filldraw[black] (-1.05,0.15) circle (2pt) node[anchor=north]{$f$};
\filldraw[black] (-0.6,0.85) circle (2pt) node[anchor=north]{$e$};
\filldraw[black] (0.3,0.15) circle (2pt) node[anchor=north]{$g$};
\filldraw[black] (-0.3,0.15) circle (2pt) node[anchor=north]{$h$};
\filldraw[black] (0,-0.35) circle (2pt) node[anchor=north]{$i$};
\end{tikzpicture}};  & 
\node (110) {\begin{tikzpicture}
\draw[black, dashed] (0,0) circle (40pt);
\draw[black, thick][-] (1.4,0.1) .. controls (0.9,0.2) .. (0.3,0.15)
                                                .. controls (0.5,0.6) .. (0.7,1.2);
\draw[black, thick][-] (-1.4,0.1) .. controls (-0.9,0.2) .. (-0.3,0.15)
                                                .. controls (-0.5,0.6) .. (-0.7,1.2);
\draw[black, thick][-] (-0.7,-1.24) .. controls (-0.3,-0.55) and (-0.2,-0.4) .. (0,-0.35)
                                                    .. controls (0.2,-0.4) and (0.3,-0.55) .. ( (0.7,-1.24);
\filldraw[black] (-0.5,-0.93) circle (2pt) node[anchor=east]{$a$};
\filldraw[black] (0.5,-0.93) circle (2pt) node[anchor=west]{$b$};
\filldraw[black] (1.05,0.15) circle (2pt) node[anchor=north]{$c$};
\filldraw[black] (0.6,0.85) circle (2pt) node[anchor=east]{$d$};
\filldraw[black] (-1.05,0.15) circle (2pt) node[anchor=north]{$f$};
\filldraw[black] (-0.6,0.85) circle (2pt) node[anchor=west]{$e$};
\filldraw[black] (0.3,0.15) circle (2pt) node[anchor=north]{$g$};
\filldraw[black] (-0.3,0.15) circle (2pt) node[anchor=north]{$h$};
\filldraw[black] (0,-0.35) circle (2pt) node[anchor=north]{$i$};
\end{tikzpicture}}; \\
};
\draw[blue,thick][->] (000.east) -- (001.west); 
\draw[blue,thick][->] (000.east) -- (010.west); 
\draw[blue,thick][->] (000.east) -- (100.west); 
\draw[blue,thick][->] (011.east) -- (111.west); 
\draw[blue,thick][->] (101.east) -- (111.west); 
\draw[blue,thick][->] (110.east) -- (111.west); 
\draw[blue,thick][->] (001.east) -- (011.west); 
\draw[blue,thick][->] (001.east) -- (101.west); 
\draw[blue,thick][->] (010.east) -- (011.west); 
\draw[blue,thick][->] (010.east) -- (110.west); 
\draw[blue,thick][->] (100.east) -- (101.west); 
\draw[blue,thick][->] (100.east) -- (110.west); 
\end{tikzpicture}
\caption{Local smoothings of a triangulated diagram, connected by the third Reidemeister move}
\label{fig9}
\end{figure}

\end{document}